\documentclass[11pt,reqno]{amsart}

\usepackage{amsmath, amsfonts, amsthm, amssymb, stmaryrd, color, cite}

\usepackage{hyperref}
\newtheorem{theorem}{Theorem}[section]
\newtheorem{lemma}{Lemma}[section]
\newtheorem{proposition}{Proposition}[section]

\theoremstyle{definition}
\newtheorem{definition}{Definition}[section]

\theoremstyle{remark}
\newtheorem{remark}{Remark}[section]

\numberwithin{equation}{section}
\allowdisplaybreaks

\def\be{\begin{equation}}
\def\en{\end{equation}}
\def\bs{\begin{split}}
\def\es{\end{split}}
\def\ba{\begin{align}}
\def\ea{\end{align}}

\author[Z. Qiu]{Zhaoyang Qiu}
\address{School of Mathematics and Statistics, Huazhong University of Science and Technology, Wuhan, 430074, China.}
\email{zhqmath@163.com }

\title[Invariant measure for stochastic CH-NS equations]
{Invariant measure for 2D stochastic Cahn-Hilliard-Navier-Stokes equations}


\date{\today}

\begin{document}
\begin{abstract}
Using the Maslowski and Seidler method, the existence of invariant measure for 2-dimensional stochastic Cahn-Hilliard-Navier-Stokes equations with multiplicative noise is proved in state space $L_x^2\times H^1$, working with the weak topology. Also, the existence of global pathwise solution is investigated using the stochastic compactness argument.
\end{abstract}

\maketitle
\section{Introduction}

The study of dynamical behaviour of solutions to SPDEs is a major subject, especially in geophysical flow, climate dynamics, gene regulation system, chemical reaction systems. Here we consider the existence of invariant measure to the 2-dimensional stochastic Cahn-Hilliard-Navier-Stokes(CH-NS) equations with multiplicative noise in a smooth bounded domain $\mathcal{D}$,
\begin{eqnarray}\label{Equ1.1}
\left\{\begin{array}{ll}
\!\!du-(\nu_0\triangle u-(u\cdot \nabla)u-\nabla p+\mu\nabla \phi)dt=h(u,\nabla\phi)d\mathcal{W},\\
\!\!d\phi+(u\cdot\nabla)\phi dt-\nu_2\triangle \mu dt=0,\\
\!\!\mu=-\nu_1\triangle \phi+ \kappa f(\phi),\\
\!\!\nabla\cdot u=0,\\
\end{array}\right.
\end{eqnarray}
which describe the motion of two immiscible mixture fluid, consisting of the Navier-Stokes equations governing the fluid velocity and a convective Cahn-Hilliard equation which is responsible for the order parameter standing for the relative concentration of one of the fluids. For further physical background, see \cite{GG,Gurtin,gio, Heida, Blesgen}.  

In equations \eqref{Equ1.1}, $u, p$ and $\phi$ denote the velocity, pressure and the phase parameter, respectively. $\nu_0$ is the viscosity of the fluid, here we assume that $\nu_0\equiv1$. Here $\mu$ stands for the chemical potential of binary mixtures which is the variational derivative of the quantity (free energy functional):
\begin{eqnarray*}
E(\phi)=\int_{\mathcal{D}}\frac{\nu_{1}}{2}|\nabla\phi|^{2}+\kappa F(\phi)dx,
\end{eqnarray*}
where $F(\phi)=\int_{0}^{\phi}f(r)dr$, the two physical parameters $\nu_{1}, \kappa$ describe the interaction between two phases. Especially, the constant $\nu_{1}>0$ is related to the thickness of the interface between the two fluids. $\nu_2>0$ is the mobility constant. In physical background, a representative example of $F$ is logarithmic type, that is,
$$F(s)=c_0[(1+s)\ln (1+s)+(1-s)\ln(1-s)]-c_1s^2, c_1>c_0>0,$$
where $s\in (-1,1)$. Usually, we use a polynomial approximation of the type $F(\phi) = C_1\phi^4-C_2\phi^2$ taking place of the type of logarithmic, where $C_1, C_2$ are two positive constants.

The system \eqref{Equ1.1} is equipped with the following natural no-flux boundary conditions
\begin{align}\label{1.2}
\frac{\partial \phi}{\partial \mathbf{n}}\bigg|_{\partial\mathcal{D}}=\frac{\partial \triangle\phi}{\partial \mathbf{n}}\bigg|_{\partial\mathcal{D}}=0,
\end{align}
where $\mathbf{n}$ is the outward normal on the boundary $\partial\mathcal{D}$. It is easily to check that above conditions imply
\begin{align}\label{1.3}
\frac{\partial \mu}{\partial \mathbf{n}}\bigg|_{\partial\mathcal{D}}=0,
\end{align}
 together with equation $\eqref{Equ1.1}_2$ yields the conservation form of phase parameter $\phi$, that is
$$\langle\phi(t)\rangle:=\frac{1}{|\mathcal{D}|}\int_{\mathcal{D}}\phi(t, x)dx=\langle\phi(0)\rangle,$$
for all $t>0$, where $|\mathcal{D}|$ is the Lebesgue measure of $\mathcal{D}$.
For the boundary condition of $u$, we impose the following Dirichlet boundary condition
\begin{align}\label{1.4}
u|_{\partial\mathcal{D}}=0.
\end{align}
Moreover, the initial data is
\begin{align}\label{1.5}
u(x,0)=u_0(x),~\phi(x,0)=\phi_0(x).
\end{align}

Let $(\Omega,\mathcal{F},\mathbb{P})$ be a complete probability space. $\mathcal{W}$ is a cylindrical Wiener process defined on the Hilbert space $\mathcal{H}$, which is adapted to the complete, right continuous filtration $\{\mathcal{F}_{t}\}_{t\geq 0}$. Namely, $\mathcal{W}=\sum_{k\geq 1}e_k\beta_{k}$ with $\{e_k\}_{k\geq 1}$ being the complete orthonormal basis of $\mathcal{H}$ and $\{\beta_{k}\}_{k\geq 1}$ being a sequence of independent standard one-dimensional Brownian motions. Considering an auxiliary space $\mathcal{H}_0\supset \mathcal{H}$, define by
\begin{eqnarray*}
\mathcal{H}_0=\left\{h=\sum_{k\geq 1}\alpha_k e_k: \sum_{k\geq 1}\alpha_k^2k^{-2}<\infty\right\},
\end{eqnarray*}
 with the norm $\|h\|_{\mathcal{H}_0}^2=\sum_{k\geq 1}\alpha_k^2k^{-2}$. Observe that the mapping $\Phi:\mathcal{H}\rightarrow\mathcal{H}_0$ is Hilbert-Schmidt. We also have that $\mathcal{W}\in C([0,\infty),\mathcal{H}_0)$ almost surely, see \cite{Zabczyk}.

Although the complicated phenomenon of mixture fluid is still far from being well-understood, since its importance in physics and the mathematical challenge, CH-NS equations already received widely attention both in deterministic and stochastic areas, see \cite{Mejdo2, FGG, FGK, FRS, GG, GMA, gio, Medjo, ZW}.

The study of invariant measure could be traced back to the 1930's, see \cite{oxto}, the well-known method of proving the existence of invariant measure was introduced by Krylov-Bogoliubov which was limited to the state space of Markov process lay in the locally compact spaces. Then, the relative theory was extended to case of the infinite dimensional Banach spaces, see \cite{brze3,flan}.  Both of the cases highly require, first the semigroup $\mathbf{P}_t$ is Feller, second, find a Borel probability measure $\xi$ on the state space $X$ and $T_0$ such that the set of measures
\begin{align}\label{1.6}
\mathcal{M}=\left\{\frac{1}{T}\int_{0}^{T}\mathbf{P}_t^*\xi dt; T>T_0\right\}
\end{align}
is tight where $\mathbf{P}_t^*$ is the dual semigroup, see \cite{Zabczyk1}, which requires us to find a auxiliary set compactly embedding into the state space.

First, in the background of SPDEs, the weak solution generally possesses the property of continuous in state space under the weak topology, second, in our case, even if for the phase parameter $\phi$, we have the auxiliary set $H^2\hookrightarrow H^1$ compactly, the high-order nonlinear term $f(\phi)$ and the special constitution cause the difficulty to get the bound for $\int_{0}^{t}\mathbb{E}\|\triangle\phi\|_{L_x^2}^2ds\sim O(t)$ for $t$ large. These two reasons push us to make us of the method developed by Maslowski and Seidler\cite{Mas} extending the Krylov-Bogoliubov theory to work with the weak topology which simplifies the proof of \eqref{1.6}. The method is widely occupied recently, see \cite{brze2} for beam equation and nonlinear wave equation, \cite{brze} for the Navier-Stokes equation in unbounded domain, \cite{bess} for the damped 2D Euler equation.

In order to construct the transition semigroup to 2-dimensional stochastic CH-NS system \eqref{Equ1.1}-\eqref{1.5}, the existence of global pathwise solution as a cornerstone will be shown by a stochastic compactness approach instead of the Galerkin-type approximation convergence method applied by \cite{Medjo, brec}. The specific strategy is, following the Yamada-Watanabe theorem, to verify the existence of martingale solution (weak in probability sense), then the known uniqueness result as well as the Gy\"{o}ngy-Krylov characterization allows us to recover the compactness. Here, we impose the following two conditions on noise intensity operator $h$: there exist two constants $K_0, K_1$ such that

{\bf C.1} $\|h(u,\nabla \phi)\|^2_{L_{2}(\mathcal{H};X)}\leq K_0(1+\|(u, \nabla \phi)\|^2_{X})$,

{\bf C.2} $\|h(u_1,\nabla \phi_1)\|^2_{L_{2}(\mathcal{H};X)}\leq K_1(1+\|(u_1-u_2, \nabla (\phi_1-\phi_2))\|^2_{X})$,\\
where $L_{2}(\mathcal{H},X)$ denotes the collection of Hilbert-Schmidt operators, the set of all linear operators $G$ from $\mathcal{H}$ to the Banach space $X$, with the norm $\|G\|_{L_{2}(\mathcal{H},X)}^2=\sum_{k\geq 1}\|Ge_k\|_{X}^2$.

Required by the tightness of \eqref{1.6} on space $(L_x^2\times H^1)_w$, the quantity
$$\int_{0}^{t}\mathbb{E}\|\nabla\phi\|_{L_x^2}^2ds$$
at most has a linear growth with respect to time $t$. Unlike the coupled SPDEs with the ''symmetric'' construction, here it seems impossible to acquire this quantity directly by taking inner product in equation $\eqref{Equ1.1}_2$. In addition, the high-order nonlinear term $f(\phi)$ also makes the estimate challenging which forces us to consider the most physically specific case $f(\phi)=C_1\phi^3-C_2\phi$, see Remark \ref{rem3.2}. To overcome these difficulties, two extra conditions are imposed on constant $K_0$ and the mobility constant $\nu_2$, that is,

{\bf C.3} the fixed constant $K_0$ in ${\bf C. 1}$ safisfies $CK_0\leq 2$, where $C=C(\mathcal{D})$ is constant,

{\bf C.4} the mobility constant $\nu_2$ is large.\\
For further technique details, see the Step 2 of Lemma \ref{lem5.2}.

The rest of paper goes as follows. We present the weak formula of the system in Section 2. In Section 3 we mainly focus on the global existence and uniqueness of the pathwise solution, also the necessary estimates are given. In section 4, we provide the auxiliary result which is crucial in both establishing the existence of pathwise solution and the sequentially weakly continuous of transition semigroup. The existence of invariant measure is proved in Section 5. An Appendix is also included afterwards to state the results that are frequently used in the paper.

\section{Weak formulation}
 At the beginning, we provide some preliminaries. Throughout the paper, we use the notation 
$$L_t^pX:=L^p(0,T; X),~~ L_\omega^pX:=L^p(\Omega;X),$$ 
$$C^\alpha_t X:=C^\alpha([0,T];X),~~ W_t^{\alpha,p}X:=W^{\alpha, p}(0,T; X)$$
for $\alpha\in (0,1], p\geq 1$, where $X$ is the Banach space and $\Omega$ is the sample space. Here, we use the fractional order Sobolev space with respect to time $t$, since the evolution equations forced by noise only have $\alpha$-order H\"{o}lder continuous in $t$ for $\alpha\in (0,\frac{1}{2})$, see \cite{ANR,FD}.  Denote $X'$ as the dual of Banach space $X$, $\mathbb{E}$ as the mathematical expectation. The notation $\langle \cdot,\cdot\rangle$ defines the dual between the space $X$ and its dual space $X'$.

Let $X_w$ be the Hilbert space $X$ with the weak topology. Introduce
$$C([0,T];X_w):={\rm the~ space ~of}~ X ~{\rm valued ~weakly ~continuous ~function},$$
endowed with the weak topology such that the mapping
$$u\mapsto \langle u,h\rangle$$
is continuous for any $h\in X$, which is a quasi-Polish space, for more details see \cite{brze2}, and
$$(L_t^pX)_w:={\rm the~ space~} L^p(0,T; X)~{\rm with~ the~ weak~ topology}.$$

We say a functional $f: X\rightarrow \mathbb{R}$ is sequentially weakly continuous if $x_n\rightarrow x$ weakly star in $X$, then $f(x_n)\rightarrow f(x)$, where $X$ is the Banach space, and we use $C_b(X_w)$ to stand for the bounded weakly continuous function.

Denote $L_x^p, p\geq 1$ the usual space of $p$-order Lebesgue integrable function on space domain $\mathcal{D}$ and
$$H^{s,p}:=\left\{u\in L_x^p; \frac{\partial^i u}{\partial x}\in L_x^p, i=1,\cdots, s\right\},$$
endowed with the norm
$$\|u\|^p_{H^{s,p}}=\|u\|^p_{L_x^p}+\sum_{i=1}^s\|\partial^i u\|^p_{L_x^p}.$$

Next, introduce the spaces of the fluid velocity $u$, denote
$$\mathcal{\mathbf{V}}=\{u\in C_c^\infty(\mathcal{D}); {\rm div} u=0~ {\rm in}~ \mathcal{D}\}.$$
Define by $\mathrm{L}_x^p$ the closure of $\mathbf{V}$ with respect to the norm of $L_x
^p$, and $\mathrm{H}^s$ the closure of $\mathbf{V}$ with respect to the norm of $H^s$.

The norm $\|(u,v)\|_{L_x^2}:=\|u\|_{\mathrm{L}_x^2}+\|v\|_{L_x^2}$ for $u\in \mathrm{L}_x^2, v\in  L_x^2$.

Define the Stokes operator by
$$A_0u=-P\triangle u,~{\rm for ~all}~u\in D(A_0):=\left\{u\in \mathrm{H}^2\cap \mathrm{L}_x^2\right\}$$
where $P$ is the Helmholtz-Leray projection. By the definition, we could have the operator $A_0$ is positive and self-adjoint on $\mathrm{L}_x^2$, and $A_0^{-1}$ is compact. Therefore, there exists a sequence $\{e_k\}_{k\geq 1}$ being an orthonormal basis for $\mathrm{L}_x^2$ of the Stokes operator $A_0$.

Also introduce the operator $A_1\phi=-\triangle \phi$ for any
$$\phi\in D(A_1)=\left\{\phi\in H^2(\mathcal{D}):\partial_n\phi|_{\partial\mathcal{D}}=0\right\}.$$
The Poincar\'{e} inequality implies the $H^2$ norm is equivalent to the norm $\|A_1(\cdot)\|_{L^2_x}+|\langle\cdot\rangle|$.

To define the variational setting, we next define $(b_{0}, b_{1}, b_{2}): D(A_0)\times D(A_0)\rightarrow \mathrm{L}^2_x, ~L_x^2\times D(A_1)\rightarrow \mathrm{L}_x^2, ~D(A_0)\times D(A_1)\rightarrow L_x^2$ as the bilinear operators such that
\begin{eqnarray*}
&&\langle b_{0}(u,v),w\rangle=\int_{\mathcal{D}}(u\cdot \nabla v)\cdot wdx=B_{0}(u,v,w),\\
&&\langle b_{1}(\mu,\phi),w\rangle=\int_{\mathcal{D}}\mu(\nabla \phi\cdot w) dx=B_{1}(\mu,\phi, w),\\
&&\langle b_{2}(u,\phi),\rho\rangle=\int_{\mathcal{D}}(u\cdot \nabla \phi)\cdot \rho dx=B_{2}(u,\phi,\rho).
\end{eqnarray*}
Note that,
$$b_0(u,u)=Pu\cdot \nabla u, ~b_1(\mu, \phi)=P\mu\nabla \phi$$
and for $u\in \mathrm{L}_x^2, v\in \mathrm{H}^1, \phi\in H^1$
\begin{align}\label{2.1*}
B_{0}(u,v,v)=0, ~B_{2}(u,\phi,\phi)=0.
\end{align}
With all these in hand, we could rewrite system \eqref{Equ1.1} into the following form
\begin{align}\label{E2.1}
\left\{\begin{array}{ll}
\!\!du+A_0udt+b_0(u,u)dt-b_1(\nu_1A_1\phi, \phi)dt=h(u, \nabla\phi)d\mathcal{W},\\
\!\!d\phi+\nu_2A_1\mu dt+b_2(u,\phi)dt=0,\\
\!\!\mu=\nu_1A_1\phi+\kappa f(\phi),
\end{array}\right.
\end{align}
understanding on space $(\mathrm{H}^1)'\times (H^1)'$ (the weak sense).
\begin{remark}Since $\nabla F(\phi)=f(\phi)\nabla \phi$, so $\mu\nabla \phi=-\nu_1\triangle \phi\cdot \nabla \phi+\kappa\nabla F(\phi)$. The term $\nabla F(\phi)$ could be absorbed into the pressure. Hence, we could replace $b_1(\mu,\phi)$ by $b_1(\nu_1A_1\phi, \phi)$.
\end{remark}

In order to handle the stochastic term, the following Burkholder-Davis-Gundy inequality is essential, for an $X$-valued predictable process $\varpi\in L^{2}(\Omega;L^{2}_{loc}([0,\infty),L_{2}(\mathcal{H},X)))$  by taking $\varpi_{k}=\varpi e_{k}$, it holds
\begin{eqnarray*}
\mathbb{E}\left[\sup_{t\in [0,T]}\left\|\int_{0}^{t}\varpi d\mathcal{W}\right\|_{X}^{p}\right]\leq c_{p}\mathbb{E}\left(\int_{0}^{T}\|\varpi\|_{L_{2}(\mathcal{H},X)}^{2}dt\right)^{\frac{p}{2}}
=c_{p}\mathbb{E}\left(\int_{0}^{T}\sum_{k\geq 1}\|\varpi_k\|_{X}^{2}dt\right)^{\frac{p}{2}},
\end{eqnarray*}
for any $p\geq1$.

\section{Existence, uniqueness and the necessary estimates}

In this section, we give the existence, uniqueness and the necessary estimates of the solution to system \eqref{Equ1.1}-\eqref{1.5}. To simplify the notation, we denote
$$\mathbb{H}:=\mathrm{L}_x^2\times H^1, \widetilde{\mathbb{H}}:=\mathrm{H}^1\times H^2.$$
\begin{definition}\label{def2.1*} (Martingale solution) Let $\Lambda$ be a Borel probability measure on $\mathbb{H}$ with $\int_{\mathbb{H}}|x|^p d\Lambda\leq C$ for a constant $C$. $(\Omega, \mathcal{F}, \{\mathcal{F}_t\}_{t\geq 0}, \mathbb{P}, \mathcal{W}, u,\phi)$ is a martingale solution to the system (\ref{Equ1.1})-(\ref{1.4}) with the initial data $\Lambda$ if the following conditions hold:

1. $(\Omega, \mathcal{F}, \{\mathcal{F}_{t}\}_{t\geq 0}, \mathbb{P})$ is a stochastic basis and  $\{\mathcal{F}_{t}\}_{t\geq 0}$ is complete right-continuous filtration, $\mathcal{W}$ is a cylindrical Wiener process,

2. $(u,\phi)$ is the $\mathcal{F}_{t}$ progressive measurable processes with
\begin{align*}
u\in L_\omega^p(L^\infty_t\mathrm{L}_x^2\cap L^2_t\mathrm{H}^1), ~\phi\in L_\omega^p(L^\infty_tH^1\cap L^2_tH^2),
\end{align*}

3. $\Lambda=\mathbb{P}\circ (u_0,\phi_0)^{-1}$,

4. for any $v\in \mathrm{H}^1, \bar{v}\in H^1$, $t\in [0,T]$, $(u,\phi)$ satisfies $\mathbb{P}${\rm-a.s.}
\begin{align}
\left\{\begin{array}{ll}
\!\!(u(t),v)+\int_{0}^{t}\langle A_0 u,v\rangle+B_0(u,u,v)-B_1(\nu_1A_1\phi, \phi, v)ds\nonumber\\ \qquad\qquad\qquad\qquad=(u_0, v)+\int_{0}^{t}(h(u, \nabla\phi),v)d\mathcal{W},\\
\!\!(\phi(t), \bar{v})+\int_{0}^{t} \nu_2\langle A_1\mu, \bar{v}\rangle+ B_2(u,\phi,\bar{v})ds=(\phi_0, \bar{v}),
\end{array}\right.
\end{align}
with
\begin{align*}
\mu=\nu_1A_1\phi+\kappa f(\phi),~\mathbb{P}\mbox{-a.s.}~ {\rm in} ~\mathcal{D}\times (0,T).
\end{align*}
\end{definition}

\begin{definition}\label{def2.1} (Pathwise solution) Let $(\Omega, \mathcal{F}, \mathbb{P})$ be the fixed probability space. $(u,\phi)$ is a global pathwise solution to the system (\ref{Equ1.1})-(\ref{1.5}) if the following conditions hold:

1. $(u,\phi)$ is $\mathcal{F}_t$ progressive measurable processes such that
\begin{align*}
u\in L_\omega^p(L^\infty_t\mathrm{L}_x^2\cap L^2_t\mathrm{H}^1), ~\phi\in L_\omega^p(L^\infty_tH^1\cap L^2_tH^2),
\end{align*}

2. for any $v\in \mathrm{H}^1, \bar{v}\in H^1$, $t\in [0,T]$, $(u,\phi)$ satisfies $\mathbb{P}${\rm-a.s.}
\begin{align}
\left\{\begin{array}{ll}
\!\!(u(t),v)+\int_{0}^{t}\langle A_0 u,v\rangle+B_0(u,u,v)-B_1(\nu_1A_1\phi, \phi, v)ds\nonumber\\ \qquad\qquad\qquad\qquad=(u_0, v)+\int_{0}^{t}(h(u, \nabla\phi),v)d\mathcal{W},\\
\!\!(\phi(t), \bar{v})+\int_{0}^{t} \nu_2\langle A_1 \mu, \bar{v}\rangle+ B_2(u,\phi,\bar{v})ds=(\phi_0, \bar{v}),
\end{array}\right.
\end{align}
with
\begin{align*}
\mu=\nu_1A_1\phi+\kappa f(\phi),~\mathbb{P}\mbox{-a.s.}~ {\rm in} ~\mathcal{D}\times (0,T).
\end{align*}
\end{definition}

\begin{definition}\label{def2.2} (Uniqueness) If $(u_1,\phi_1),(u_2,\phi_2)$ are two solutions of system (\ref{Equ1.1})-(\ref{1.5}). Suppose that the initial data $(u_1(0),\phi_1(0))$ and $(u_2(0),\phi_2(0))$ satisfy
$$(u_1(0),\phi_1(0))=(u_2(0),\phi_2(0)),$$
then the pathwise uniqueness holds if
\begin{eqnarray*}
\mathbb{P}\Big\{(u_{1}(t,x), \phi_1(t,x))=(u_{2}(t,x),\phi_2(t,x));~\forall t\in[0,T]\Big\}=1.
\end{eqnarray*}
\end{definition}

\begin{theorem}\label{thm3.1} Assume that the initial data $(u_0,\phi_0)$ satisfies $\|(u_0,\phi_0)\|_{\mathbb{H}}\leq C_3$ for constant $C_3>0$, the operator $h$ satisfies the Conditions {\bf C.1} and {\bf C.2}. Then, there exists a unique global pathwise solution to system (\ref{Equ1.1})-(\ref{1.5}) in the sense of Definitions \ref{def2.1},\ref{def2.2}.
\end{theorem}
\begin{remark} The proof of Theorem \ref{thm3.1} was given in pioneering work \cite{Medjo}, for deterministic case see \cite{GG}. To identify the limit, the author implemented the Galerkin-type approximation convergence in mean square introduced by \cite{brec}. Here, unlike \cite{Medjo}, we use the stochastic compactness argument to establish the existence of pathwise solution.
\end{remark}

\begin{proof}  For the proof of existence, the crucial technique details is the same with the result built in Section 4, therefore, here we only give a simplify proof .

Existence. \underline{Step 1}. For the existence of finite-dimensional Galerkin approximate solutions, we refer the readers to \cite{Medjo}, and also the sequence of approximate solutions $(u_n, \phi_n)$ satisfies the following uniform a priori estimates in $n$ for all $p\geq 2$
\begin{align}\label{3.1*}
u_n\in L_\omega^p(L^\infty_t\mathrm{L}_x^2\cap L^2_t\mathrm{H}^1), ~\phi_n\in L_\omega^p(L^\infty_tH^1\cap L^2_tH^2), ~\mu_n\in L_\omega^p(L^2_tH^1).
\end{align}

\begin{remark}\label{rem3.2} The proof of the a priori estimates \eqref{3.1*} is actually simpler than \eqref{5.13*}. Here, we do not have to handle the term $(F(\phi), 1)$, since it has the bound from below. Moreover, we could use the stochastic Gronwall lemma to absorb the term $\mathbb{E}\int_{0}^{t}\|\nabla \phi\|_{L^2_x}^2ds$ arising from noise term, getting the upper bound with the exponential growth of $t$ which is not suitable for proving the tightness of \eqref{1.6} on the space $\mathbb{H}_w$.
\end{remark}

To order to obtain the tightness of the set of probability measures induced by the law of $(u_n, \phi_n)$, we need the following necessary estimate:

there exists constant $C$ such that for any $\alpha\in (0,\frac{1}{2}), \tilde{\alpha}\in (0,1]$
\begin{align}\label{3.1}
\mathbb{E}\|u_n\|^2_{C_t^\alpha(\mathrm{H}^{-1,\frac{3}{2}})}+\mathbb{E}\|\phi_n\|^2_{C_t^{\tilde{\alpha}}(H^{-1})}\leq C,
\end{align}
where $C$ is independent of $n$.

To simplify the notation, we use $(u,\phi)$ replacing $(u_n, \phi_n)$. Note that, for a.s. $\omega$, and for any $\epsilon>0$, there exists $t_1,t_2\in [0,T]$ such that
$$\sup_{t_0\neq t'_0}\frac{\left\|\int_{t_0}^{t'_0}Pu\cdot \nabla uds\right\|_{\mathrm{H}^{-1,\frac{3}{2}}}}{|t'_0-t_0|^\alpha}\leq \frac{\left\|\int_{t_1}^{t_2}Pu\cdot \nabla uds\right\|_{\mathrm{H}^{-1,\frac{3}{2}}}}{|t_2-t_1|^\alpha}+\epsilon.$$
Taking $t=\frac{3}{2}, s=2, r=2$ in Lemma \ref{lem6.1} and using H\"{o}lder's inequality, for $\alpha\leq 1$,
\begin{align}\label{3.2}
\mathbb{E}\sup_{t,t'\in [0,T]}\frac{\left\|\int_{t'}^{t}Pu\cdot \nabla uds\right\|_{\mathrm{H}^{-1,\frac{3}{2}}}}{|t-t'|^\alpha}&\leq \mathbb{E}\frac{\left\|\int_{t_1}^{t_2}Pu\cdot \nabla uds\right\|_{\mathrm{H}^{-1,\frac{3}{2}}}}{|t_2-t_1|^\alpha}+\epsilon\nonumber\\
&\leq \mathbb{E}\frac{\int_{t_1}^{t_2}\|Pu\cdot \nabla u\|_{\mathrm{H}^{-1,\frac{3}{2}}}ds}{|t_2-t_1|^\alpha}+\epsilon\nonumber\\
&\leq \mathbb{E}\frac{\int_{t_1}^{t_2}\|u\|_{\mathrm{H}^1} \|\nabla u\|_{\mathrm{H}^{-1}}ds}{|t_2-t_1|^\alpha}+\epsilon\nonumber\\
&\leq \frac{\left(\mathbb{E}\sup_{t\in [0,T]}\|u\|_{\mathrm{L}_x^2}^2\right)^\frac{1}{2}\left(\mathbb{E}\left[\int_{t_1}^{t_2}\|\nabla u\|_{\mathrm{H}^{-1}}ds\right]^2\right)^\frac{1}{2}}{|t_2-t_1|^\alpha}+\epsilon\nonumber\\
&\leq |t_2-t_1|^{1-\alpha}\left(\mathbb{E}\sup_{t\in [0,T]}\|u\|_{\mathrm{L}_x^2}^2\right)^\frac{1}{2}\left(\mathbb{E}\int_{t_1}^{t_2}\| u\|^2_{\mathrm{H}^{1}}ds\right)^\frac{1}{2}+\epsilon\nonumber\\
&\leq C.
\end{align}
Similarly, one can estimate
\begin{align}
\mathbb{E}\sup_{t,t'\in [0,T]}\frac{\left\|\int_{t'}^{t}P(A_1\phi\cdot \nabla \phi )ds\right\|_{\mathrm{H}^{-1,\frac{3}{2}}}}{|t-t'|^\alpha}&\leq \mathbb{E}\frac{\left\|\int_{t_1}^{t_2}A_1\phi\cdot \nabla \phi ds\right\|_{H^{-1,\frac{3}{2}}}}{|t_2-t_1|^\alpha}+\epsilon\nonumber\\
&\leq \mathbb{E}\frac{\int_{t_1}^{t_2}\|A_1\phi\|_{H^{-1}} \|\nabla \phi\|_{H^{1}}ds}{|t_2-t_1|^\alpha}+\epsilon\nonumber\\
&\leq |t_2-t_1|^{1-\alpha}\!\!\left(\!\mathbb{E}\sup_{t\in [0,T]}\|\nabla \phi\|_{L_x^2}^2\!\!\right)^\frac{1}{2}\!\!\left(\!\mathbb{E}\int_{t_1}^{t_2}\|\nabla \phi\|^2_{H^{1}}ds\!\!\right)^\frac{1}{2}+\epsilon\nonumber\\
&\leq C,
\end{align}
also, taking $t=1, s=3, r=2$ in Lemma \ref{lem6.1} and using the embedding $L_x^2\hookrightarrow H^{-1,3}$,
\begin{align}
\mathbb{E}\sup_{t,t'\in [0,T]}\frac{\left\|\int_{t'}^{t}u\cdot \nabla \phi ds\right\|_{H^{-1}}}{|t-t'|^\alpha}&\leq \mathbb{E}\frac{\left\|\int_{t_1}^{t_2}u\cdot \nabla \phi ds\right\|_{H^{-1}}}{|t_2-t_1|^\alpha}+\epsilon\nonumber\\
&\leq \mathbb{E}\frac{\int_{t_1}^{t_2}\|u\|_{\mathrm{H}^{-1,3}} \|\nabla \phi\|_{H^{1}}ds}{|t_2-t_1|^\alpha}+\epsilon\nonumber\\
&\leq \mathbb{E}\frac{\int_{t_1}^{t_2}\|u\|_{\mathrm{L}_x^2} \|\nabla \phi\|_{H^{1}}ds}{|t_2-t_1|^\alpha}+\epsilon\nonumber\\
&\leq |t_2-t_1|^{1-\alpha}\left(\mathbb{E}\sup_{t\in [0,T]}\|u\|_{\mathrm{L}_x^2}^2\right)^\frac{1}{2}\!\!\left(\mathbb{E}\int_{t_1}^{t_2}\|\nabla \phi\|^2_{H^{1}}ds\right)^\frac{1}{2}+\epsilon\nonumber\\
&\leq C.
\end{align}
By the Burkholder-Davis-Gundy inequality, the Condition {\bf C.1} and \eqref{3.1*}, for $v\in \mathrm{H}^1$
\begin{align}\label{3.5}
\mathbb{E}\left(\sup_{t_0\neq t'_0}\frac{\left|\int_{t_0}^{t'_0}(h(u,\nabla\phi), v)d\mathcal{W}\right|}{|t'_0-t_0|^\alpha}\right)^p
\leq &~\mathbb{E}\frac{\left(\int_{t_1}^{t_2}\|h(u,\nabla\phi)\|_{L_2(\mathcal{H};L_x^2)}^2 \|v\|_{\mathrm{L}_x^2}^2
dt\right)^\frac{p}{2}}{|t_2-t_1|^{\alpha p}}+\epsilon^p\nonumber\\
\leq &~\frac{CK_0^\frac{p}{2}(t_2-t_1)^\frac{p}{2}\mathbb{E}(1+\|(u,\nabla\phi)\|_{L_t^\infty L_x^2})^p}{|t_1-t_2|^{\alpha p}}+\epsilon^p\nonumber\\
\leq &~CK_0^\frac{p}{2}|t_2-t_1|^{\left(\frac{1}{2}-\alpha\right)p}+\epsilon^p\leq C,
\end{align}
for any $\alpha<\frac{1}{2}$. For the linear term, by \eqref{3.1*} we easily get
\begin{align}\label{3.7}
 \mathbb{E}\left\|\int_{0}^{t}A_0 u ds\right\|^2_{C_t^1(\mathrm{H}^{-1})}\leq C,~  \mathbb{E}\left\|\int_{0}^{t}A_1 \mu ds\right\|^2_{C_t^1(H^{-1})}\leq C.
\end{align}
Taking into account of estimates (\ref{3.2})-(\ref{3.7}), we obtain the desired bound (\ref{3.1}).

\underline{Step 2}. Next, we could use the bounds (\ref{3.1*})(\ref{3.1}) to establish the tightness of the probability measure set $\mathcal{L}^n$ induced by the law of $(u_n,\phi_n)$, the further details of the proof is same with the argument as Lemma \ref{lem4.1} which is implemented for the solutions $(u_l, \phi_l)$ corresponding to initial data $(u_{0,l}, \phi_{0,l})$, then, we could identify the limit using the technique details we present in Step 2 of the proof of Proposition \ref{pro4.1}, establishing the existence of global martingale solution which is weak in both probability sense and PDE sense in the sense of Definition \ref{def2.1*}.

\underline{Step 3}. To obtain the existence of global pathwise solution which is strong in probability sense, we shall recover the compactness of approximate sequences $(u_n,\phi_n)$ on the original space $(\Omega,\mathcal{F},\mathbb{P})$ applying the Gy\"{o}ngy-Krylov characterization, for the proof see Appendix. Then, passing the limit, the global pathwise solution follows.
\end{proof}

\section{Continuous dependence of initial data}

Let $(u_{0,l},\phi_{0,l})$ be a sequence of initial data with
\begin{align}\label{4.1*}
\|(u_{0,l},\phi_{0,l})\|_{\mathbb{H}}\leq C_3 ~{\rm for ~all}~ l\in \mathbb{N}^+,
\end{align}
according to the Theorem \ref{thm3.1}, there exists a sequence global pathwise solutions $(u_l, \phi_l)$ to system (\ref{Equ1.1})-(\ref{1.4}) corresponding to initial data $(u_{0,l}, \phi_{0,l})$ and the fixed stochastic basis $(\Omega,\mathcal{F},\{\mathcal{F}_{t}\}_{t\geq0},\mathbb{P}, \mathcal{W})$, meanwhile, these solutions still enjoy the following uniform bounds
\begin{align}\label{4.1}
\mathbb{E}\left(\|u_l\|^2_{C_t^\alpha(\mathrm{H}^{-1,\frac{3}{2}})}+\|\phi_l\|^2_{C_t^{\tilde{\alpha}}(H^{-1})}\right)\leq C
\end{align}
and
\begin{align}\label{4.2}
\mathbb{E}\left(\|(u_l,\phi_l)\|^p_{L_t^\infty\mathbb{H}}+\|(u_l,\phi_l)\|^p_{L^2_t\widetilde{\mathbb{H}}}+\|\mu_l\|_{L_t^2H^1}^p\right)\leq C,
\end{align}
for any $p\geq 2, \alpha\in (0,\frac{1}{2}), \tilde{\alpha}\in (0,1]$, where constant $C$ is independent of $l$.

This section is mainly to show the following continuous dependence result:
\begin{proposition}\label{pro4.1} Assume that initial data $(u_{0,l}, \phi_{0,l})$ is an $\mathbb{H}$-valued sequence which convegences weakly to $(u_{0},\phi_{0})\in \mathbb{H}$, and $\eqref{4.1*}$ holds for certain constant $C_3$. Then, there exist a new probability space
$(\widetilde{\Omega}, \widetilde{\mathcal{F}}, \widetilde{\mathbb{P}})$, a new subsequence $(\tilde{u}_{l_k},\tilde{\phi}_{l_k}, \tilde{\mu}_{l_k}, \widetilde{\mathcal{W}}_{l_k})$ and the processes $(\tilde{u},\tilde{\phi}, \tilde{\mu}, \widetilde{\mathcal{W}})$, such that
\begin{align}
(\tilde{u}_{l_k},\tilde{\phi}_{l_k}, \tilde{\mu}_{l_k}, \widetilde{\mathcal{W}}_{l_k}) ~&{\rm and} ~(u_{l_k},\phi_{l_k}, \mu_{l_k}, \mathcal{W}) ~{\rm have ~ the ~same ~joint~ distribution~ in} ~\mathcal{X},\label{4.3*}\\
(\tilde{u},\tilde{\phi}) ~&{\rm and} ~(u,\phi) ~{\rm have ~ the ~same ~joint~ distribution~ in} ~\mathcal{X}\label{4.4*}
\end{align}
and
\begin{align}\label{4.5*}
(\tilde{u}_{l_k},\tilde{\phi}_{l_k},\tilde{\mu}_{l_k}, \widetilde{\mathcal{W}}_{l_k})\rightarrow (\tilde{u},\tilde{\phi},\tilde{\mu}, \widetilde{\mathcal{W}}), ~{\rm in~ the~ topology ~of~ \mathcal{X}}, ~\widetilde{\mathbb{P}}\mbox{-a.s.}
\end{align}
where the $(u,\phi)$ is the pathwise solution with the initial data $(u_0,\phi_0)$. Moreover, we have for all $p\geq 1$
\begin{align}\label{4.5}
\widetilde{\mathbb{E}}\left(\sup_{t\in [0,T]}\|(\tilde{u}_{l_k},\tilde{\phi}_{l_k})\|_{\mathbb{H}}^{2p}\right)+
\widetilde{\mathbb{E}}\left(\int_{0}^{T}\|(\tilde{u}_{l_k},\tilde{\phi}_{l_k})\|^2_{\widetilde{\mathbb{H}}}+\|\tilde{\mu}_{l_k}\|_{H^1}^2dt\right)^p\leq C.
\end{align}
In addition, $(\widetilde{\Omega}, \widetilde{\mathcal{F}}, \widetilde{\mathbb{P}}, \tilde{u},\tilde{\phi}, \widetilde{\mathcal{W}})$ is a martingale solution of the system \eqref{Equ1.1} in the sense of Definition \ref{def2.1*}.
\end{proposition}

\begin{remark} For the Galerkin approximate solutions $(u_n,\phi_n)$ relative to the finite dimensional initial data $(u_{0,n},\phi_{0,n})$ and the fixed stochastic basis $(\Omega, \mathcal{F}, \mathbb{P}, \mathcal{W})$, we also have the same results as \eqref{4.3*}, \eqref{4.5*}: there exist a new probability space
$(\widehat{\Omega}, \widehat{\mathcal{F}}, \widehat{\mathbb{P}})$, a new subsequence $(\hat{u}_{n_k},\hat{\phi}_{n_k}, \hat{\mu}_{n_k}, \widehat{\mathcal{W}}_{n_k})$ and the processes $(\hat{u},\hat{\phi}, \hat{\mu}, \widehat{\mathcal{W}})$, such that
\begin{align*}
(\hat{u}_{n_k},\hat{\phi}_{n_k}, \hat{\mu}_{n_k}, \widehat{\mathcal{W}}_{n_k}) ~&{\rm and} ~(u_{n_k},\phi_{n_k}, \mu_{n_k}, \mathcal{W}) ~{\rm have ~ the ~same ~joint~ distribution~ in} ~\mathcal{X},
\end{align*}
and
\begin{align*}
(\hat{u}_{n_k},\hat{\phi}_{n_k},\hat{\mu}_{n_k}, \widehat{\mathcal{W}}_{n_k})\rightarrow (\hat{u},\hat{\phi},\hat{\mu}, \widehat{\mathcal{W}}), ~{\rm in~ the~ topology ~of~ \mathcal{X}}, ~\widetilde{\mathbb{P}}\mbox{-a.s.}.
\end{align*}
\end{remark}

The key point of showing above continuous dependence result is to acquire the tightness of the probability measures induced by the law of $(u_l, \phi_l)$. Define the set
$$\mathcal{X}=\mathcal{X}_{u,\phi}\times \mathcal{X}_{\mu}\times \mathcal{X}_\mathcal{W},$$
where
$$\mathcal{X}_{u,\phi}:=C([0,T]; \mathbb{H}_w)\cap L_t^2\mathbb{H}\cap (L_t^2\widetilde{\mathbb{H}})_w,~~ \mathcal{X}_{\mu}:=L_t^2L_x^2\cap(L^2_tH^1)_w , ~~  \mathcal{X}_\mathcal{W}:=C_t (\mathcal{H}_0)$$
and $\mathbb{H}_w$ denotes the spaces $\mathrm{L}_x^2,H^1$ with the weak topology of $\mathrm{L}_x^2,H^1$ respectively. Let $\Gamma$ be the smallest topology on $\mathcal{X}$ such that four natural embedding from $\mathcal{X}$ are continuous.

Define the measure set
$$\mathcal{L}^l(B)=\mathbb{P}\{(u_l,\phi_l,\mu_l, \mathcal{W})\in B\},$$
for $B\in \mathcal{B}(\mathcal{X})$ which is the smallest $\sigma$-algebra containing the family $\Gamma$.
\begin{lemma}\label{lem4.1}
The measure set $\mathcal{L}^l$ is tight on space $(\mathcal{X}, \Gamma)$.
\end{lemma}
\begin{proof} It is enough to show that every measure set
\begin{align}
&\mathcal{L}^l_u(B_1):=\mathbb{P}\{u_l\in B_1\}, ~\mathcal{L}^l_\phi(B_2):=\mathbb{P}\{\phi_l\in B_2\},\nonumber\\
& \mathcal{L}_\mathcal{W}(B_3):=\mathbb{P}\{\mathcal{W}\in B_3\},  \mathcal{L}_\mu^l(B_4):=\mathbb{P}\{\mu_l\in B_4\},\nonumber
\end{align}
is tight on $\mathcal{X}_u, \mathcal{X}_\phi, \mathcal{X}_\mu, \mathcal{X}_\mathcal{W} $ respectively, for $B_1\in \mathcal{B}(\mathcal{X}_u), B_2\in \mathcal{B}(\mathcal{X}_\phi), B_3\in \mathcal{B}(\mathcal{X}_\mathcal{W}), B_4\in \mathcal{B}(\mathcal{X}_\mu)$, where $\mathcal{X}_u:=C([0,T]; (\mathrm{L}_x^2)_w)\cap L_t^2\mathrm{L}_x^2\cap (L^2_t\mathrm{H}^1)_w$,~~ $\mathcal{X}_\phi:=C([0,T]; H^1_w)\cap L^2_tH^{1}\cap (L^2_tH^2)_w$.

{\bf Claim} 1. The probability measure set $\mathcal{L}^l_{u}$ is tight on $\mathcal{X}_u$.

Define the set
\begin{align*}
B^u_{1,K}&=\left\{u_{l}\in L_t^2\mathrm{H}^1\cap C_t^\alpha(\mathrm{H}^{-1,\frac{3}{2}}):\|u_{l}\|_{ L_t^2\mathrm{H}^1}+\|u_{l}\|_{C_t^\alpha(\mathrm{H}^{-1,\frac{3}{2}})}\leq K\right\},\\
B^u_{2,K}&=\left\{u_{l}\in L_t^\infty \mathrm{L}_x^2\cap C_t^\alpha(\mathrm{H}^{-1,\frac{3}{2}}): \|u_{l}\|_{L_t^\infty \mathrm{L}_x^2}+\|u_{l}\|_{C_t^\alpha(\mathrm{H}^{-1,\frac{3}{2}})}\leq K\right\}.
\end{align*}
The Aubin-Lions compact embedding, see Lemma \ref{lem4.5}
$$L_t^2H^1\cap C_t^\alpha(H^{-1,\frac{3}{2}})\hookrightarrow L^2_tL^2_x,~L^\infty_tL_x^2\cap C_t^\alpha(H^{-1,\frac{3}{2}})\hookrightarrow C([0,T];(L_x^2)_w)$$
implies that the set $B^u_{1,K}, B^u_{2,K}$ is relatively compact in $L^2_t\mathrm{L}^2_x, C([0,T];(L_x^2)_w)$. Considering (\ref{4.1}), (\ref{4.2}), by the Chebyshev inequality we have
\begin{align*}
\mathcal{L}_u^{l}((B^u_{1,K}\cap B^u_{2,K})^c)&\leq \mathbb{P}\left(\|u_{l}\|_{ L_t^2\mathrm{H}^1}>\frac{K}{2}\right)+\mathbb{P}\left(\|u_{l}\|_{ {C_t^\alpha(\mathrm{H}^{-1,\frac{3}{2}})}}>\frac{K}{2}\right)\nonumber\\
&\quad+\mathbb{P}\left(\|u_{l}\|_{ L^\infty_t\mathrm{L}_x^2}>\frac{K}{2}\right)+\mathbb{P}\left(\|u_{l}\|_{ {C_t^\alpha(\mathrm{H}^{-1,\frac{3}{2}})}}>\frac{K}{2}\right)\nonumber\\
&\leq \frac{C}{K}\mathbb{E}\left(\|u_{l}\|_{ L_t^2\mathrm{H}^1}+\|u_{l}\|_{ L^\infty_t\mathrm{L}_x^2}+\|u_{l}\|_{C_t^\alpha(\mathrm{H}^{-1,\frac{3}{2}})}\right)\leq \frac{C}{K}.
\end{align*}
Note that, by the Alaoglu-Banach theorem, the set
\begin{eqnarray*}
B^u_{3,K}:=\left\{u_{l}\in L_t^{2} \mathrm{H}^1: \|u_{l}\|_{L_t^{2} \mathrm{H}^1}\leq K\right\}
\end{eqnarray*}
is relatively compact on path space $(L_t^{2} \mathrm{H}^1)_w$. The bound (\ref{4.2}) and the Chebyshev inequality again, yield
$$\mathcal{L}^l_{u}( (B_{3,K}^u)^c)\leq \frac{C}{K}.$$
We obtain the tightness of set $\mathcal{L}^l_{u}$.

{\bf Claim} 2. The probability measure set $\mathcal{L}^l_{\phi}$ is tight on $\mathcal{X}_\phi$.

In a similar way to the {\bf Claim} 1, define the set
\begin{align*}
B^\phi_{1,K}&=\left\{\phi_{l}\in L_t^2H^2\cap C_t^\alpha(H^{-1}):\|\phi_{l}\|_{ L_t^2H^1}+\|\phi_{l}\|_{C_t^\alpha(H^{-1})}\leq K\right\},\\
B^\phi_{2,K}&=\left\{\phi_{l}\in L_t^\infty H^1\cap C_t^\alpha(H^{-1}): \|\phi_{l}\|_{L_t^\infty H^1}+\|\phi_{l}\|_{C_t^\alpha(H^{-1})}\leq K\right\},\\
B^\phi_{3,K}&=\left\{\phi_{l}\in L_t^{2} H^2: \|\phi_{l}\|_{L_t^{2} H^2}\leq K\right\}.
\end{align*}
Applying the Aubin-Lions compact embedding
$$L_t^2H^2\cap C_t^{\alpha}H^{-1}\hookrightarrow L_t^2H^{1}, L_t^\infty H^1\cap C_t^\alpha(H^{-1})\hookrightarrow C([0,T]; H^1_w),$$
we infer that $B^\phi_{1,K}, B^\phi_{2,K}$ is relatively compact in $L_t^2H^{1}, C([0,T]; H^1_w)$ respectively. Also, by the Alaoglu-Banach theorem, the set $B^\phi_{3,K}$ is relatively compact in $(L_t^2H^{2})_w$.

By bounds \eqref{4.1}, \eqref{4.2}, we deduce that the probability measure set $\mathcal{L}^l_{\phi}$ is tight. 

{\bf Claim} 3. The probability measure set $\mathcal{L}^l_{\mu}$ is tight on $\mathcal{X}_\mu$.

Since 
\begin{align*}
&\mathbb{E}\left(\int_{0}^{T}\int_{0}^{T}\frac{\|f(\phi(t_1))-f(\phi(t_2))\|^p_{H^{-1,\frac{3}{2}}}}{|t_1-t_2|^{1+\alpha p}}dt_1 dt_2\right)^\frac{1}{p}\nonumber\\
&\leq C\mathbb{E}\left(\int_{0}^{T}\int_{0}^{T}\frac{\|\phi(t_1)-\phi(t_2)\|^p_{H^{-1}}\|\phi^2(t_1)+\phi^2(t_2)\|_{H^1}^p}{|t_1-t_2|^{1+\alpha p}}dt_1 dt_2\right)^\frac{1}{p}\nonumber\\
&\leq C\mathbb{E}\Bigg(\sup_{t_1,t_2\in [0,T]}\frac{\|\phi(t_1)-\phi(t_2)\|_{H^{-1}}}{|t_1-t_2|^{\frac{1}{p}+\alpha }}\nonumber\\
&\qquad\quad\times\int_{0}^{T}\int_{0}^{T}\|(\phi(t_1), \phi(t_2))\|^p_{L_x^6}\|\nabla \phi(t_1), \nabla\phi(t_2)\|_{L_x^3}^pdt_1dt_2^\frac{1}{p}\Bigg)\nonumber\\
&\leq C\left[\mathbb{E}\left(\sup_{t_1,t_2\in [0,T]}\frac{\|\phi(t_1)-\phi(t_2)\|_{H^{-1}}}{|t_1-t_2|^{\frac{1}{p}+\alpha }}\right)^2\right]^\frac{1}{2}\nonumber\\
&\quad\times\left[\mathbb{E}\left(\int_{0}^{T}\int_{0}^{T}\|(\phi(t_1), \phi(t_2))\|^p_{H^1}\|(\nabla \phi(t_1), \nabla\phi(t_2))\|_{H^2}^pdt_1dt_2\right)^\frac{2}{p}\right]^\frac{1}{2}
\end{align*}
taking $p=2, \alpha=\frac{1}{2}$, by bounds \eqref{4.1}, \eqref{4.2}, the H\"{o}lder inequality, we have
$$\mathbb{E}\|\mu_l\|_{W^{\frac{1}{2}, 2}(0,T; H^{-3})}\leq C,$$
together with bound \eqref{4.2}, using the same argument as above, we could infer that $\mathcal{L}_\mu^l$ is tight on $\mathcal{X}_\mu$.

Since the sequence $\mathcal{W}$ is only one element and thus, the set $\mathcal{L}_\mathcal{W}$ is weakly compact. Lemma \ref{lem4.1} follows.
\end{proof}

\begin{proof}[Proof of Proposition \ref{pro4.1}] \underline{Step 1}. With the tightness of probability measure set $\mathcal{L}^l$ in hand, from the Skorokhod-Jakubowski representation theorem (Theorem \ref{thm4.2}), we obtain \eqref{4.3*} directly, and
\begin{align*}
(\tilde{u}_{l_k},\tilde{\phi}_{l_k},\tilde{\mu}_{l_k}, \widetilde{\mathcal{W}}_{l_k})\rightarrow (\tilde{u},\tilde{\phi},\tilde{\mu}, \widetilde{\mathcal{W}}), ~{\rm in~ the~ topology ~of~ \mathcal{X}}, ~\widetilde{\mathbb{P}}\mbox{-a.s.}
\end{align*}
In particular, the chemical potential $\mu$ is a nonlinear term, we need to verify that $\tilde{\mu}=-\nu_1\triangle \tilde{\phi}+\kappa f(\widetilde{\phi})$, a.e. on $(0,T)\times \mathcal{D}$. Note that, by the mean-value theorem, there exists $\theta\in (0,1)$ such that
\begin{align*}
\|f(\tilde{\phi}_{l_k})-f(\widetilde{\phi})\|_{L_t^2 L_x^2}&\leq C\left\||\theta \tilde{\phi}+(1-\theta)\tilde{\phi}_{l_k}|^2(\tilde{\phi}_{l_k}-\tilde{\phi})\right\|_{L_t^2 L_x^2}\nonumber\\
&\leq C\left\|\||\theta \tilde{\phi}+(1-\theta)\tilde{\phi}_{l_k}|^2\|_{L_x^3}\|\tilde{\phi}_{l_k}-\tilde{\phi}\|_{L_x^6}\right\|_{L_t^2}\nonumber\\
&\leq C\|(\tilde{\phi}_{l_k}, \tilde{\phi})\|_{L_t^\infty H^1}^2\|\tilde{\phi}_{l_k}-\tilde{\phi}\|_{L_t^2H^1}.
\end{align*}
The result follows from the fact that $\|\tilde{\phi}_{l_k}-\tilde{\phi}\|_{L_t^2H^1}\rightarrow 0,~ \widetilde{\mathbb{P}}\mbox{-a.s.}$ as $k\rightarrow\infty$. Since $(\tilde{u}_{l_k},\tilde{\phi}_{l_k}, \tilde{\mu}_{l_k})$  and $(u_{l_k},\phi_{l_k}, \mu_{l_k})$ have the same joint distribution, we obtain \eqref{4.5} from \eqref{4.2}.

\underline{Step 2}. We show that $(\tilde{u},\tilde{\phi},\tilde{\mu}, \widetilde{\mathcal{W}})$ also satisfies the system \eqref{E2.1}. Again, since the sequences $(\tilde{u}_{l_k},\tilde{\phi}_{l_k}, \tilde{\mu}_{l_k}, \widetilde{\mathcal{W}}_{l_k})$  and $(u_{l_k},\phi_{l_k}, \mu_{l_k}, \mathcal{W})$ have the same joint distribution, we have $(\tilde{u}_{l_k},\tilde{\phi}_{l_k}, \tilde{\mu}_{l_k}, \widetilde{\mathcal{W}}_{l_k})$ satisfies the system (\ref{E2.1}), the proof is standard, see \cite{ww}. Next, we identify the limit by passing $k\rightarrow \infty$.

For any $v\in \mathrm{H}^2$, decomposing
$$B_0(\tilde{u}_{l_k},  \tilde{u}_{l_k}, v)-B_0(\tilde{u}, \tilde{u}, v)=B_0( \tilde{u}, \tilde{u}_{l_k}-\tilde{u}, v)+B_0( \tilde{u}_{l_k}-\tilde{u}, \tilde{u}_{l_k}, v).$$
By the convergence $\tilde{u}_{l_k}\rightarrow \tilde{u}$ in $L_t^2\mathrm{L}_x^2, ~\widetilde{\mathbb{P}}\mbox{-a.s.}$ and Lemma \ref{lem6.1}, we have $\widetilde{\mathbb{P}} \mbox{-a.s.}$
$$\int_{0}^{t}|B_0( \tilde{u}_{l_k}-\tilde{u}, \tilde{u}_{l_k}, v)|ds\leq \|v\|_{\mathrm{H}^2}\|\tilde{u}_{l_k}-\tilde{u}\|_{L_t^2\mathrm{L}_x^2}\| \tilde{u}_{l_k}\|_{L_t^2\mathrm{H}^1}\rightarrow 0.$$
Furthermore, using the bound $\tilde{u}\in L_t^2\mathrm{H}^1$ and the convergence $\tilde{u}_{l_k}\rightarrow \tilde{u}$ in $(L_t^2\mathrm{H}^1)_w, \widetilde{\mathbb{P}}\mbox{-a.s.}$ we infer that
$$\int_{0}^{t}B_0( \tilde{u}, \tilde{u}_{l_k}-\tilde{u}, v)ds\rightarrow 0.$$
By Lemma \ref{lem6.1}, the H\"{o}lder inequality and the bound (\ref{4.5}),
\begin{align}\label{4.8}
\widetilde{\mathbb{E}}\left(\int_{0}^{T}B_0(\tilde{u}_{l_k}, \tilde{u}_{l_k}, v)dt\right)^p&\leq \widetilde{\mathbb{E}}\left(\int_{0}^{T}\|\tilde{u}_{l_k}\|_{\mathrm{L}_x^2}\| \tilde{u}_{l_k}\|_{\mathrm{H}^1} \|v\|_{\mathrm{H}^{1,3}}dt\right)^p\nonumber\\
&\leq \|v\|^p_{\mathrm{H}^2}\widetilde{\mathbb{E}}\left[\sup_{t\in [0,T]}\|\tilde{u}_{l_k}\|_{\mathrm{L}_x^2}^p\left(\int_{0}^{T}\| \tilde{u}_{l_k}\|_{\mathrm{H}^1}dt\right)^p\right]\nonumber\\
&\leq CT^\frac{p}{2}\|v\|^p_{\mathrm{H}^2}\widetilde{\mathbb{E}}\left(\sup_{t\in [0,T]}\|\tilde{u}_{l_k}\|_{\mathrm{L}_x^2}^{2p}\right)\widetilde{\mathbb{E}}\left(\int_{0}^{T}\| \tilde{u}_{l_k}\|^2_{\mathrm{H}^1}dt\right)^p\nonumber\\
&\leq C.
\end{align}

Since
$$A_1\phi\cdot \nabla \phi=\nabla \left(\frac{|\nabla\phi|^2}{2}\right)-{\rm div} (\nabla \phi \otimes \nabla\phi),$$
then, we have
\begin{align*}
&\quad B_1(\nu_1A_1\tilde{\phi}_{l_k},\tilde{ \phi}_{l_k}, v)-B_1(\nu_1A_1\tilde{\phi}, \tilde{\phi}, v)\nonumber\\&=\int_{\mathcal{D}}\nu_1(\nabla \tilde{\phi}_{l_k}\otimes \nabla \tilde{\phi}_{l_k}-\nabla \tilde{\phi}\otimes \nabla \tilde{\phi}): \nabla vdx \nonumber\\
&=\int_{\mathcal{D}}\nu_1((\nabla \tilde{\phi}_{l_k}-\nabla \tilde{\phi})\otimes \nabla \tilde{\phi}_{l_k}+\nabla \tilde{\phi}\otimes (\nabla \tilde{\phi}_{l_k}-\nabla \tilde{\phi})): \nabla vdx.
\end{align*}
By the convergence $\tilde{\phi}_{l_k}\rightarrow \tilde{\phi}$ in $L_t^2H^1$, $\widetilde{\mathbb{P}}\mbox{-a.s.}$ and the bound (\ref{4.5}), we have $\widetilde{\mathbb{P}}\mbox{-a.s.}$
\begin{align*}
\int_{0}^{t}\int_{\mathcal{D}}|((\nabla \tilde{\phi}_{l_k}-\nabla \tilde{\phi})\otimes \nabla \tilde{\phi}_{l_k}):\nabla v| dxds&\leq \|\nabla \tilde{\phi}_{l_k}-\nabla \tilde{\phi}\|_{L_t^2L_x^2}\|\nabla \tilde{\phi}_{l_k}\|_{L^2_tL_x^6}\|\nabla v\|_{\mathrm{L}_x^{3}}\nonumber\\
&\leq \|\nabla \tilde{\phi}_{l_k}-\nabla \tilde{\phi}\|_{L_t^2L_x^2}\|\nabla \tilde{\phi}_{l_k}\|_{L^2_tH^1}\|\nabla v\|_{\mathrm{L}_x^{3}}\rightarrow 0.
\end{align*}
Also by the convergence $\tilde{\phi}_{l_k}\rightarrow \tilde{\phi}$ in $(L_t^2H^2)_w$ and the bound $\tilde{\phi}\in L_t^2H^2$, $\widetilde{\mathbb{P}}\mbox{-a.s.}$ we have
\begin{align*}
\int_{0}^{t}\int_{\mathcal{D}}\nabla \tilde{\phi}\otimes (\nabla \tilde{\phi}_{l_k}-\nabla \tilde{\phi})): \nabla vdxds\rightarrow 0,~ \widetilde{\mathbb{P}} \mbox{-a.s.}
\end{align*}
By Lemma \ref{lem6.1}, the H\"{o}lder inequality and the bound (\ref{4.5}), we have
\begin{align}\label{4.10}
\widetilde{\mathbb{E}}\left(\int_{0}^{T}B_1(\nu_1A_1\tilde{\phi}_{l_k},\tilde{ \phi}_{l_k}, v)dt\right)^p&\leq C\|v\|^p_{\mathrm{H}^{1,3}}\widetilde{\mathbb{E}}\left(\int_{0}^{T}\|A_1\tilde{\phi}_{l_k}\|_{H^{-1}}\|\nabla\tilde{ \phi}_{l_k}\|_{H^1}dt\right)^p\nonumber\\
&\leq CT^\frac{p}{2}\|v\|^p_{\mathrm{H}^{1,3}}\widetilde{\mathbb{E}}\left(\sup_{t\in [0,T]}\|\nabla\tilde{\phi}_{l_k}\|_{L^2}^{2p}\right)\widetilde{\mathbb{E}}\left(\int_{0}^{T}\| \nabla\tilde{\phi}_{l_k}\|^2_{H^1}dt\right)^p\nonumber\\ &\leq C.
\end{align}
Applying the same argument, we could infer that for any $\bar{v}\in H^1$
$$\int_{0}^{t}B_2(\tilde{u}_{l_k}, \tilde{\phi}_{l_k},\bar{v})-B_2(\tilde{u},\tilde{\phi}, \bar{v})ds\rightarrow 0, ~\widetilde{\mathbb{P}}\mbox{-a.s.}$$
and
\begin{align}\label{4.10*}
\widetilde{\mathbb{E}}\left(\int_{0}^{T}B_2(\tilde{u}_{l_k}, \tilde{\phi}_{l_k},\bar{v})dt\right)^p\leq C.
\end{align}

Finally, we pass the limit in stochastic term. Using the Condition {\bf C.2}
$$\|h(\tilde{u}_{l_k}, \nabla\tilde{\phi}_{l_k})-h(\tilde{u},\nabla\tilde{\phi})\|_{L_2(\mathcal{H};L_x^2)}\leq K_1\|(\tilde{u}_{l_k}-\tilde{u}, \nabla(\tilde{\phi}_{l_k}-\tilde{\phi}))\|_{L^2_x},$$
together with the Proposition $\ref{pro4.1}$ \eqref{4.5*} yields
$$\|h(\tilde{u}_{l_k}, \nabla\tilde{\phi}_{l_k})-h(\tilde{u},\nabla\tilde{\phi})\|_{L_2(\mathcal{H};L_x^2)}\rightarrow 0,~ \mbox{a.e.}~ (t, \omega)\in [0,T]\times \widetilde{\Omega}.$$
Combining the convergence and the bound (\ref{4.5}), we have from the Vitali convergence theorem (see Theorem \ref{thm4.1}),
$$h(\tilde{u}_{l_k}, \nabla\tilde{\phi}_{l_k})\rightarrow h(\tilde{u},\nabla\tilde{\phi}),~{\rm in}~ L^p_\omega L_t^2L_2(\mathcal{H};L_x^2),$$
which implies the convergence in probability in $L_t^2L_2(\mathcal{H};L_x^2)$. Moreover, from Proposition $\ref{pro4.1}$ \eqref{4.5*}, we have
$$\widetilde{\mathcal{W}}_{l_k}\rightarrow \widetilde{\mathcal{W}},~ {\rm in~ probability~ in} ~C_t(\mathcal{H}_0).$$
Then, applying the Lemma \ref{lem4.6}, we infer that
$$\int_{0}^{t}h(\tilde{u}_{l_k}, \nabla\tilde{\phi}_{l_k})d \widetilde{\mathcal{W}}_{l_k}\rightarrow \int_{0}^{t}h(\tilde{u}, \nabla\tilde{\phi})d \widetilde{\mathcal{W}},$$
in probability in $L_t^2L_2(\mathcal{H};L_x^2)$. Using the It\^{o} isometry and the Condition {\bf C.1}, we have
\begin{align}\label{4.11}
\widetilde{\mathbb{E}}\left|\int_{0}^{T}h(\tilde{u}_{l_k}, \nabla\tilde{\phi}_{l_k})d\widetilde{\mathcal{W}}_{l_k}\right|^2
&=C\widetilde{\mathbb{E}}\int_{0}^{T}\|h(\tilde{u}_{l_k}, \nabla\tilde{\phi}_{l_k})\|^2_{L_2(\mathcal{H};L^2_x)}dt\nonumber\\
&\leq CK_0\left(1+\widetilde{\mathbb{E}}\int_{0}^{T}\|\tilde{u}_{l_k}\|_{\mathrm{L}_x^2}^2+\|\nabla\tilde{\phi}_{l_k}\|_{L_x^2}^2dt\right)\nonumber\\
&\leq C.
\end{align}

 Proposition \ref{pro4.1} \eqref{4.5*} gives $(\tilde{u}_{l_k}, \tilde{\phi}_{l_k}) \rightarrow (\tilde{u}, \tilde{\phi})$, $\widetilde{\mathbb{P}}\mbox{-a.s.}$ in $C([0,T]; \mathbb{H}_w)$, using the dominated convergence theorem we get for all $t\in [0,T]$, $\widetilde{\mathbb{P}}\mbox{-a.s.}$
 $$\lim_{k\rightarrow \infty}\int_{0}^{t}|(\tilde{u}_{l_k}-\tilde{u}, v)|^2ds\rightarrow 0,~\lim_{k\rightarrow \infty} \int_{0}^{t}|(\tilde{\phi}_{l_k}-\tilde{\phi}, \bar{v})|^2ds\rightarrow 0,$$
which combines with the bound (\ref{4.5}), we obtain
\begin{align*}
\widetilde{\mathbb{E}}\int_{0}^{T}|(\tilde{u}_{l_k}-\tilde{u}, v)|^2dt\rightarrow 0,
~\widetilde{\mathbb{E}}\int_{0}^{T}|(\tilde{\phi}_{l_k}-\tilde{\phi}, \bar{v})|^2dt\rightarrow 0.
\end{align*}
By the same way, we have
\begin{align*}
\widetilde{\mathbb{E}}|(\tilde{u}_{l_k}(0)-\tilde{u}(0), v)|^2\rightarrow 0,~
\widetilde{\mathbb{E}}|(\tilde{\phi}_{l_k}(0)-\tilde{\phi}(0), \bar{v})|^2\rightarrow 0.
\end{align*}

For the linear term, since $(\tilde{u}_{l_k}, \tilde{\mu}_{l_k})\rightarrow (\tilde{u},\tilde{\mu})$, $\widetilde{\mathbb{P}}\mbox{-a.s.}$ in $(L_t^2(\mathrm{H}^1\times H^1))_w$, we infer that
$$\int_{0}^{t}\langle A_0\tilde{u}_{l_k}, v\rangle ds\rightarrow \int_{0}^{t}\langle A_0\tilde{u}, v\rangle ds, ~\int_{0}^{t}\langle A_1\tilde{\mu}_{l_k}, \bar{v}\rangle ds\rightarrow \int_{0}^{t}\langle A_1\tilde{\mu}, \bar{v}\rangle ds.$$
In addition, we have
\begin{align}
&\widetilde{\mathbb{E}}\left|\int_{0}^{T}\langle A_0\tilde{u}_{l_k}, v\rangle dt\right|^2\leq C\|v\|^2_{\mathrm{H}^1}\widetilde{\mathbb{E}}\int_{0}^{T}\| \nabla\tilde{u}_{l_k}\|^2_{\mathrm{L}_x^2} dt\leq C,\label{4.14*}\\
&\widetilde{\mathbb{E}}\left|\int_{0}^{T}\nu_2\langle A_1\tilde{\mu}_{l_k}, \bar{v}\rangle dt\right|^2\leq C\|\bar{v}\|^2_{H^1}\widetilde{\mathbb{E}}\int_{0}^{T}\| \nabla\tilde{\mu}_{l_k}\|^2_{L_x^2} dt\leq C.\label{4.15*}
\end{align}
All the constants $C$ in above are independent of $l_k$.

Next, define the functional
\begin{align*}
\mathbb{F}_1(\tilde{u}_{l_k},\tilde{ \phi}_{l_k},\widetilde{\mathcal{W}}_{l_k},v )&=(\tilde{u}_{l_k}(0), v)-\int_{0}^{t}\langle A_0\tilde{u}_{l_k}, v\rangle ds-\int_{0}^{t}B_0(\tilde{u}_{l_k},  \tilde{u}_{l_k}, v)ds\nonumber\\&\quad+\int_{0}^{t}B_1(\nu_1A_1\tilde{\phi}_{l_k},\tilde{ \phi}_{l_k}, v)ds+\left\langle\int_{0}^{t}h(\tilde{u}_{l_k}, \nabla\tilde{\phi}_{l_k})d \widetilde{\mathcal{W}}_{l_k}, v\right\rangle,\nonumber\\
\mathbb{F}_2(\tilde{u}_{l_k},\tilde{ \phi}_{l_k},\bar{v} )&=(\tilde{\phi}_{l_k}(0), \bar{v})-\int_{0}^{t}\nu_2\langle A_1\tilde{\mu}_{l_k}, \bar{v}\rangle ds-\int_{0}^{t}B_2(\tilde{u}_{l_k},  \tilde{\phi}_{l_k}, \bar{v})ds
\end{align*}
and
\begin{align*}
\overline{\mathbb{F}}_1(\tilde{u},\tilde{ \phi},\widetilde{\mathcal{W}},v )&=(\tilde{u}(0), v)-\int_{0}^{t}\langle A_0\tilde{u}, v\rangle ds-\int_{0}^{t}B_0(\tilde{u},  \tilde{u}, v)ds\nonumber\\&\quad+\int_{0}^{t}B_1(\nu_1A_1\tilde{\phi},\tilde{ \phi}, v)ds+\left\langle\int_{0}^{t}h(\tilde{u}, \nabla\tilde{\phi})d \widetilde{\mathcal{W}}, v\right\rangle,\nonumber\\
\overline{\mathbb{F}}_2(\tilde{u},\tilde{ \phi},\bar{v} )&=(\tilde{\phi}(0), \bar{v})-\int_{0}^{t}\nu_2\langle A_1\tilde{\mu}, \bar{v}\rangle ds-\int_{0}^{t}B_2(\tilde{u},  \tilde{\phi}, \bar{v})ds.
\end{align*}
Combining the bounds (\ref{4.8})-(\ref{4.15*}) and the convergence properties, using the Vitali convergence theorem, we conclude that for any set $B\in [0,T]$
\begin{align}
&\widetilde{\mathbb{E}}\int_{0}^{T}1_{B}(\mathbb{F}_1(\tilde{u}_{l_k},\tilde{ \phi}_{l_k},\widetilde{\mathcal{W}}_{l_k},v )-\overline{\mathbb{F}}_1(\tilde{u},\tilde{ \phi},\widetilde{\mathcal{W}},v ))dt\rightarrow 0,\label{4.14}\\
&\widetilde{\mathbb{E}}\int_{0}^{T}1_{B}(\mathbb{F}_2(\tilde{u}_{l_k},\tilde{ \phi}_{l_k},\widetilde{\mathcal{W}}_{l_k},\bar{v} )-\overline{\mathbb{F}}_2(\tilde{u},\tilde{ \phi},\widetilde{\mathcal{W}},\bar{v} ))dt\rightarrow 0, \label{4.15}
\end{align}
where $1_{\{\cdot\}}$ stands for the indicator function.

Since $(\tilde{u}_{l_k}, \tilde{ \phi}_{l_k})$ satisfies the system \eqref{E2.1} in the weak sense, that is, $$((\tilde{u}_{l_k},v), (\tilde{ \phi}_{l_k},\bar{v}))=(\mathbb{F}_1(\tilde{u}_{l_k},\tilde{ \phi}_{l_k},\widetilde{\mathcal{W}}_{l_k},v ), \mathbb{F}_2(\tilde{u}_{l_k},\tilde{ \phi}_{l_k},\widetilde{\mathcal{W}}_{l_k},\bar{v} )), ~\widetilde{\mathbb{P}}\mbox{-a.s.}$$
From (\ref{4.14})-(\ref{4.15}), we conclude that for any $t\in [0,T]$, $v\in \mathrm{H}^2, \bar{v}\in H^1$ and $\widetilde{\mathbb{P}}\mbox{-a.s.}$
$$(\tilde{u},v)=\overline{\mathbb{F}}_1(\tilde{u},\tilde{ \phi},\widetilde{\mathcal{W}},v),~~ (\tilde{\phi}, \bar{v})=\overline{\mathbb{F}}_2(\tilde{u},\tilde{ \phi},\bar{v} ).$$
Finally, by the density argument, the above equality holds for any $(v, \bar{v})\in \mathrm{H}^1\times H^1$. We obtain that the $(\widetilde{\Omega}, \widetilde{\mathcal{F}}, \widetilde{\mathbb{P}}, \tilde{u},\tilde{\phi}, \widetilde{\mathcal{W}})$ is a martingale solution of the system \eqref{Equ1.1}-\eqref{1.4}. 

\underline{Step 3}. Finally, we show $\eqref{4.4*}$ holds. Since $(u,\phi)$ is a global pathwise solution with the initial data $(u_0, \phi_0)$ and the fixed stochastic basis $(\Omega, \mathcal{F}, \mathbb{P}, \mathcal{W})$, also $(\widetilde{\Omega}, \widetilde{\mathcal{F}}, \widetilde{\mathbb{P}}, \tilde{u},\tilde{\phi}, \widetilde{\mathcal{W}})$ is a martingale solution of the system \eqref{Equ1.1}-\eqref{1.4} with the initial data $\Lambda$, where $\Lambda=\mathbb{P}\circ (u_0,\phi_0)^{-1}$, and the solution of system \eqref{Equ1.1}-\eqref{1.4} is pathwise unique, so unique in law, we deduce that $\eqref{4.4*}$ holds.

This completes the proof.
\end{proof}

\section{Existence of invariant measure}
In this section, we show the main result: existence of invariant measure.  We say a probability measure $\xi$ on $\mathcal{B}(X_w)$ is an invariant measure if
$$\int_{X}\mathbf{P}_t\varphi d\xi=\int_{X}\varphi d\xi,~~ {\rm for~ all}~t\geq 0, ~\varphi\in C_b(X_w).$$
Note that, if the space $X$ is separable, then $\mathcal{B}(X_w)=\mathcal{B}(X)$.

For any bounded Borel function $\varphi\in \mathcal{B}_b(\mathbb{H})$, we define
\begin{align}\label{5.1*}
(\mathbf{P}_t\varphi)(U_0)=\mathbb{E}[\varphi(U(t, U_0))],~ U_0\in \mathbb{H},
\end{align}
where $U(t, U_0)=(u(t, u_0), \phi(t, \phi_0)))$ is the pathwise solution of system (\ref{Equ1.1})-\eqref{1.5} starting from the initial data $U_0=(u_0,\phi_0)$, which is a Markov process. The proof of Markov property is standard, see \cite{brze, bess}. We also have $\mathbf{P}_{t+s}\varphi(\cdot)=\mathbf{P}_{t}\mathbf{P}_{s}\varphi(\cdot)$ for $t,s\geq 0$.
\begin{lemma}\label{lem5.1} Let $\varphi:\mathbb{H}\rightarrow \mathbb{R}$ be the bounded sequentially weakly continuous function. Then, we have $$\mathbf{P}_t:C_b(\mathbb{H}_w)\rightarrow C_b(\mathbb{H}_w),$$
particularly, if $U_{0,l}\rightarrow U_0$ weakly in $\mathbb{H}$, then
$$(\mathbf{P}_t\varphi)(U_{0,l})\rightarrow (\mathbf{P}_t\varphi)(U_0).$$
\end{lemma}
\begin{proof} The proof is inspired by \cite{brze}. From the Theorem \ref{thm3.1}, corresponding to the initial data $U_{0,l}$, there exists a sequence $\mathbb{H}$-valued $\mathcal{F}_t$ progressive measurable processes $U_{l}=(u_{l}, \phi_l)$ which is the solution of system (\ref{Equ1.1})-\eqref{1.5}. Using the Proposition \ref{pro4.1}, there exist a new probability space
$(\widetilde{\Omega}, \widetilde{\mathcal{F}}, \widetilde{\mathbb{P}})$, a new subsequence $(\tilde{u}_{l_k},\tilde{\phi}_{l_k})$ and the processes $(\tilde{u},\tilde{\phi})$, such that
$$(\tilde{u}_{l_k},\tilde{\phi}_{l_k})~{\rm and} ~(u_{l_k},\phi_{l_k}), ~ (\tilde{u}, \tilde{\phi})~{\rm and}~(u,\phi) ~{\rm have ~ the ~same ~joint~ distribution~ in} ~\mathcal{X}_{u,\phi},$$
and
$$(\tilde{u}_{l_k},\tilde{\phi}_{l_k})\rightarrow (\tilde{u},\tilde{\phi}), ~{\rm in~ the~ topology ~of~ \mathcal{X}_{u,\phi}}, ~\widetilde{\mathbb{P}}\mbox{-a.s.}$$
where $\widetilde{U}=(\tilde{u},\tilde{\phi})$ satisfies the system \eqref{E2.1}.

Since $\widetilde{U}_{l_k}:=(\tilde{u}_{l_k},\tilde{\phi}_{l_k})\rightarrow \widetilde{U}$ weakly in $C([0,T]; \mathbb{H}_w)$, this together with the fact that $\varphi$ is a bounded sequentially weakly continuous function, yields
$$\varphi(\widetilde{U}_{l_k})\rightarrow \varphi(\widetilde{U}),~ {\rm in} ~\mathbb{R}, ~\widetilde{\mathbb{P}}\mbox{-a.s.}$$
Since the function $\varphi$ is bounded, we infer from the dominated convergence theorem
$$\widetilde{\mathbb{E}}\left[\varphi(\widetilde{U}_{l_k}(t;\widetilde{U}_{0,l_k}))\right]\rightarrow \widetilde{\mathbb{E}}\left[\varphi(\widetilde{U}(t, \widetilde{U}_0))\right], ~{\rm as} ~k\rightarrow \infty.$$
Since the processes $\widetilde{U}_{l_k}$ and $U_{l_k}$, $\widetilde{U}$ and $U$ have the same distribution which lead to
\begin{align}
&\widetilde{\mathbb{E}}\left[\varphi(\widetilde{U}_{l_k}(t;\widetilde{U}_{0,l_k}))\right]=\mathbb{E}\left[\varphi(U_{l_k}(t;U_{0,l_k}))\right]
=(\mathbf{P}_t\varphi)(U_{0,l_k}),\label{5.2*}\\
&\widetilde{\mathbb{E}}\left[\varphi(\widetilde{U}(t;\widetilde{U}_0))\right]=\mathbb{E}\left[\varphi(U(t;U_0))\right]=(\mathbf{P}_t\varphi)(U_{0}).\label{5.3*}
\end{align}
Therefore, by \eqref{5.2*},\eqref{5.3*}, we have
$$\lim_{k\rightarrow\infty}(\mathbf{P}_t\varphi)(U_{0,l_k})=(\mathbf{P}_t\varphi)(U_{0}).$$
Using the sub-subsequence argument, we have the original sequence
$$\lim_{k\rightarrow\infty}(\mathbf{P}_t\varphi)(U_{0,l})=(\mathbf{P}_t\varphi)(U_{0}).$$
Obviously, $\mathbf{P}_t\varphi$ from $\mathbb{H}$ into $\mathbb{R}$ is bounded. This completes the proof.
\end{proof}

\begin{lemma}\label{lem5.2} Let $U(t)$ be the solution of system \eqref{E2.1} with the initial data $U_0$. Then, there exists $T_0$ such that for every $\varepsilon>0$, there exists $R>0$ such that
$$\sup_{T>T_0}\frac{1}{T}\int_{0}^{T}(\mathbf{P}_t^*\delta_{U_0})(\mathbb{H}\setminus B_R)dt\leq \varepsilon,$$
where $B_R:=\{U(t)\in \mathbb{H}: \|U(t)\|_{\mathbb{H}}\leq R\}$.
\end{lemma}
\begin{proof} \underline{Step 1}. We first show that there exists constant $C$ independent of $t$ such that
$$\int_{0}^{t}\mathbb{E}\|U(s)\|_{\mathbb{H}}^2ds\leq C(1+t).$$
Using the It\^o formula to function $\|u\|_{\mathrm{L}^2_x}^2$, by (\ref{2.1*}), we have
\begin{align}\label{5.3}
\|u\|_{\mathrm{L}^2_x}^2+2\int_{0}^{t}\|\nabla u\|_{\mathrm{L}^2_x}^2ds&=\|u_0\|_{\mathrm{L}^2_x}^2+2\int_{0}^{t}B_1(\nu_1A_1\phi,\phi,u)ds+2\int_{0}^{t}(h(u, \nabla \phi)d\mathcal{W},u)\nonumber\\
&\quad+\int_{0}^{t}\|h(u, \nabla \phi)\|^2_{L_2(\mathcal{H};L^2_x)}ds.
\end{align}
Taking inner product with $\mu$ on both sides of equation $(\ref{E2.1})_2$, we have
\begin{align}\label{5.4}
&\nu_1\|\nabla \phi\|_{L^2_x}^2+2\kappa(F(\phi),1)+2\nu_2\int_{0}^{t}\|\nabla \mu\|_{L^2_x}^2ds+2\int_{0}^{t}B_2(u,\phi,\mu)ds\nonumber\\&=\nu_1\|\nabla \phi_0\|_{L^2_x}^2+2\kappa(F(\phi_0),1).
\end{align}
Taking the sum of (\ref{5.3})-(\ref{5.4}), using the relation
$$B_2(u,\phi,\mu)=B_1(\nu_1A_1\phi,\phi,u),$$
we have
\begin{align}\label{5.6}
&\|u\|_{\mathrm{L}^2_x}^2+\nu_1\|\nabla \phi\|_{L^2_x}^2+2\kappa(F(\phi),1)+2\int_{0}^{t}\|\nabla u\|_{\mathrm{L}^2_x}^2+\nu_2\|\nabla \mu\|_{L^2_x}^2ds\nonumber\\&=\|u_0\|_{\mathrm{L}^2_x}^2+\nu_1\|\nabla \phi_0\|_{L^2_x}^2+2\kappa(F(\phi_0),1)+2\int_{0}^{t}(h(u, \nabla \phi)d\mathcal{W},u)\nonumber\\&\quad+\int_{0}^{t}\|h(u, \nabla \phi)\|^2_{L_2(\mathcal{H};L^2_x)}ds.
\end{align}
Taking expectation on both sides of \eqref{5.6}, we arrive
\begin{align}\label{3.6**}
&\mathbb{E}\left(\|u\|_{\mathrm{L}^2_x}^2+\nu_1\|\nabla \phi\|_{L^2_x}^2+2\kappa(F(\phi),1)+2\int_{0}^{t}\|\nabla u\|_{\mathrm{L}^2_x}^2+\nu_2\|\nabla \mu\|_{L^2_x}^2ds\right)\nonumber\\&=\|u_0\|_{\mathrm{L}^2_x}^2+\nu_1\|\nabla \phi_0\|_{L^2_x}^2+2\kappa(F(\phi_0),1)+\mathbb{E}\int_{0}^{t}\|h(u, \nabla \phi)\|^2_{L_2(\mathcal{H};L^2_x)}ds.
\end{align}
The stochastic term is a square integral martingale which its expectation vanished. Indeed,
\begin{align}\label{5.7}
\mathbb{E}\int_{0}^{t}\|h(u, \nabla \phi)\|^2_{L_2(\mathcal{H};\mathrm{L}^2_x)}ds\leq K_0\mathbb{E}\int_{0}^{t}1+(\|u\|_{\mathrm{L}^2_x}^2+\|\nabla \phi\|_{L^2_x}^2)ds\leq C.
\end{align}

By \eqref{5.7} again, we get from \eqref{3.6**}
\begin{align}\label{5.8*}
&\mathbb{E}\left(\|u\|_{\mathrm{L}^2_x}^2+\nu_1\|\nabla \phi\|_{L^2_x}^2+2\kappa C_1\|\phi\|_{L_x^4}^4+2\int_{0}^{t}\|\nabla u\|_{\mathrm{L}^2_x}^2+\nu_2\|\nabla \mu\|_{L^2_x}^2ds\right)\nonumber\\&\leq\|u_0\|_{\mathrm{L}^2_x}^2+\nu_1\|\nabla \phi_0\|_{L^2_x}^2+2\kappa(F(\phi_0),1)\nonumber\\&\quad+ K_0\mathbb{E}\int_{0}^{t}1+(\|u\|_{\mathrm{L}^2_x}^2+\|\nabla \phi\|_{L^2_x}^2)ds+2\kappa C_2\mathbb{E}\|\phi\|_{L_x^2}^2\nonumber\\
&\leq \|u_0\|_{\mathrm{L}^2_x}^2+\nu_1\|\nabla \phi_0\|_{L^2_x}^2+2\kappa(F(\phi_0),1)\nonumber\\&\quad+ K_0\mathbb{E}\int_{0}^{t}1+(\|u\|_{\mathrm{L}^2_x}^2+\|\nabla \phi\|_{L^2_x}^2)ds+2\kappa C_2|\mathcal{D}|^\frac{1}{2}\mathbb{E}\|\phi\|_{L_x^4}^2\nonumber\\
&\leq \|u_0\|_{\mathrm{L}^2_x}^2+\nu_1\|\nabla \phi_0\|_{L^2_x}^2+2\kappa(F(\phi_0),1)+2\delta_0\kappa C_2|\mathcal{D}|^\frac{1}{2}\mathbb{E}\|\phi\|_{L_x^4}^4+C(\delta_0, \kappa, C_2, \mathcal{D})+K_0t\nonumber\\ &\quad+K_0\mathbb{E}\int_{0}^{t}\|u\|_{\mathrm{L}^2_x}^2+\|\nabla \phi\|_{L^2_x}^2ds,
\end{align}
where we choose $\delta_0$ small such that $2\delta_0\kappa C_2|\mathcal{D}|^\frac{1}{2}\leq 2\kappa C_1$.

The key point is to control the last term $\mathbb{E}\int_{0}^{t}\|\nabla \phi\|_{L^2_x}^2ds$. Taking inner product with $\phi$ with $(\ref{E2.1})_3$, then integral in time $t$ and taking expectation, we get
\begin{align}\label{3.9*}
\nu_1\mathbb{E}\int_{0}^{t}\|\nabla \phi\|_{L_x^2}^2ds+\kappa\mathbb{E}\int_{0}^{t}(F(\phi),1)ds&=\mathbb{E}\int_{0}^{t}(\mu, \phi)ds.
\end{align}
Note that, by $(\ref{E2.1})_3$ and the boundary condition \eqref{1.2}, we get
$$\langle \mu \rangle=\kappa\langle f(\phi)\rangle.$$
Using the  H\"{o}lder inequality and the Schwartz inequality, we get
\begin{align}\label{5.10}
\langle f(\phi)\rangle &=\frac{1}{|\mathcal{D}|}\int_{\mathcal{D}}f(\phi)dx=\frac{1}{|\mathcal{D}|}\int_{\mathcal{D}}C_1\phi^3-C_2\phi dx\nonumber\\
&\leq \frac{1}{|\mathcal{D}|}\left(C_1|\mathcal{D}|^\frac{1}{4}\int_{\mathcal{D}}\phi^4 dx^\frac{3}{4}+C_2|\mathcal{D}|^\frac{1}{2}\int_{\mathcal{D}}\phi^2 dx^\frac{1}{2}\right)\nonumber\\
&\leq \frac{C_1 \delta}{|\mathcal{D}|}\int_{\mathcal{D}}\phi^4 dx+\frac{C_2 \delta}{|\mathcal{D}|}\int_{\mathcal{D}}\phi^2 dx+\frac{C}{\delta},
\end{align}
where $\delta$ is constant selected later.

Using the H\"{o}lder inequality, the Schwartz inequality and boundary condition (\ref{1.2}), estimate (\ref{5.10}), the right hand term of \eqref{3.9*} could be controlled as follows
\begin{align}\label{3.10*}
\mathbb{E}\int_{0}^{t}(\mu, \phi)ds&\leq \mathbb{E}\int_{0}^{t}\|\mu\|_{L_x^2}\|\phi\|_{L_x^2}ds\leq \frac{1}{2\delta} \mathbb{E}\int_{0}^{t}\|\mu\|_{L_x^2}^2ds+\frac{\delta}{2}\mathbb{E}\int_{0}^{t}\|\phi\|_{L_x^2}^2ds\nonumber\\
&\leq\frac{ C}{2\delta} \mathbb{E}\int_{0}^{t}\|\nabla\mu\|_{L_x^2}^2+\|\langle\mu\rangle\|_{L_x^2}^2ds+\frac{C\delta}{2}\mathbb{E}\int_{0}^{t}\|\nabla\phi\|_{L_x^2}^2
+\|\langle\phi\rangle\|_{L_x^2}^2ds\nonumber\\
&\leq \frac{ C}{2\delta} \mathbb{E}\int_{0}^{t}\|\nabla\mu\|_{L_x^2}^2+\frac{C_1 \delta}{\sqrt{|\mathcal{D}|}}\int_{\mathcal{D}}\phi^4 dx+\frac{C(\mathcal{D})}{\delta}ds\nonumber\\&\quad+\frac{C\delta}{2}\mathbb{E}\int_{0}^{t}\|\nabla\phi\|_{L_x^2}^2
+\|\langle\phi_0\rangle\|_{L_x^2}^2ds+\frac{C_2 \delta}{|\mathcal{D}|}\mathbb{E}\int_{0}^{t}\|\nabla\phi\|_{L_x^2}^2
+\|\langle\phi_0\rangle\|_{L_x^2}^2ds.
\end{align}
Here, choosing $\delta, K_0$ small such that
$$\max\left\{\frac{C\delta}{2}+\frac{C_2 \delta^2}{2|\mathcal{D}|}+K_0, \frac{C_1 \delta}{\sqrt{|\mathcal{D}|}}\right\}\leq \min\{\nu_1, \kappa\},$$
then choosing $2\nu_2>\frac{ C}{2\delta}$.
The Dirichlet boundary condition gives
\begin{align}\label{5.12*}
\|u\|_{\mathrm{L}^2_x}\leq C\|\nabla u\|_{\mathrm{L}^2_x},
\end{align}
where constant $C$ depends on $\mathcal{D}$.
Combining (\ref{5.8*}), (\ref{3.9*}) and \eqref{3.10*}, \eqref{5.12*}, we get
\begin{align}\label{5.13}
\mathbb{E}\Bigg(\|u\|_{\mathrm{L}^2_x}^2&+C\int_{0}^{t}\|\phi\|_{L_x^4}^4ds+C\int_{0}^{t}\|\nabla \phi\|_{L^2_x}^2ds\nonumber\\&+\int_{0}^{t}(2-K_0C)\|\nabla u\|_{\mathrm{L}^2_x}^2+(2\nu_2-\frac{ C}{2\delta})\|\nabla \mu\|_{L^2_x}^2ds\Bigg)\nonumber\\
\leq &C(C_1, C_2, \kappa,\nu_1, \delta, \mathcal{D})+C(C_1, C_2, \kappa,\nu_1,\delta, \mathcal{D})t+C_2\kappa\mathbb{E}\int_{0}^{t}\|\phi\|_{L_x^2}^2ds\nonumber\\
&+\|u_0\|_{\mathrm{L}^2_x}^2+\nu_1\|\nabla \phi_0\|_{L^2_x}^2+2\kappa(F(\phi_0),1).
\end{align}
Using the Sobolev embedding $H^1\hookrightarrow L_x^4\hookrightarrow L_x^2$, we have
$$|2\kappa(F(\phi_0),1)|\leq C(\|\phi_0\|_{L_x^4}^4+\|\phi_0\|_{L_x^2}^2)\leq C(1+\|\phi_0\|_{H^1}^4).$$
According to the Condition ${\bf C. 4}$, by \eqref{5.13} and above estimate, we have
\begin{align}\label{5.14}
&\mathbb{E}\left(C\int_{0}^{t}\|\phi\|_{L_x^4}^4ds+C\int_{0}^{t}\|\phi\|_{H^1}^2ds+\int_{0}^{t}(2-K_0C)\|\nabla u\|_{\mathrm{L}^2_x}^2ds\right)\nonumber\\
\leq &C(C_1, C_2, C_3, \kappa,\nu_1, \delta, \mathcal{D})+C(C_1, C_2, \kappa,\nu_1,\delta, \mathcal{D})t+(C+C_2\kappa)\mathbb{E}\int_{0}^{t}\|\phi\|_{L_x^2}^2ds.
\end{align}
Since
\begin{align}\label{5.15}
(C+C_2\kappa)\mathbb{E}\int_{0}^{t}\|\phi\|_{L_x^2}^2ds&\leq C(C_2, \kappa, \mathcal{D})\mathbb{E}\int_{0}^{t}\|\phi\|_{L_x^4}^2ds\nonumber\\
&\leq \delta_1 C(C_2, \kappa, \mathcal{D})\mathbb{E}\int_{0}^{t}\|\phi\|_{L_x^4}^4ds+C(\delta_1, C_2, \kappa, \mathcal{D})t.
\end{align}
Choosing $\delta_1$ small, combining \eqref{5.14}, \eqref{5.15} and the Condition ${\bf C. 3}$, we have 
\begin{align*}
\int_{0}^{t}\mathbb{E}\left(\|\nabla u\|_{\mathrm{L}^2_x}^2+\| \phi\|_{H^1}^2\right)ds\leq C+Ct,
\end{align*}
where constant $C$ depends on $(C_1, C_2, C_3, \kappa,\nu_1, \nu_2, \mathcal{D}, K_0)$ but independent of $t$.

By \eqref{5.12*} again, we conclude
\begin{align}\label{5.13*}
\int_{0}^{t}\mathbb{E}\left(\| u\|_{\mathrm{L}^2_x}^2+\| \phi\|_{H^1}^2\right)ds\leq C+Ct,
\end{align}
where constant $C$ depends on $(C_1, C_2, C_3, \kappa,\nu_1, \nu_2, \mathcal{D}, K_0)$ but independent of $t$.

\underline{Step 2}. By the Chebyshev inequality and (\ref{5.13*}),
\begin{align*}
\frac{1}{T}\int_{0}^{T}(\mathbf{P}_t^*\delta_{U_0})(\mathbb{H}\setminus B_R)dt&=\frac{1}{T}\int_{0}^{T}\mathbb{P}\{\|U(t)\|_{\mathbb{H}}>R\}dt\nonumber\\
&\leq \frac{1}{R^2T}\int_{0}^{T}\mathbb{E}\|U(t)\|_{\mathbb{H}}^2dt\nonumber\\
&\leq \frac{C+CT}{R^2T}.
\end{align*}
This completes the proof.
\end{proof}

\begin{theorem}\label{thm5.1} Suppose that the Conditions ${\bf C. 1}$-${\bf C. 4}$ hold, then the transition semigroup defined by \eqref{5.1*} has an invariant measure.
\end{theorem}
\begin{proof} Combining the Lemmas \ref{lem5.1},\ref{lem5.2}, Theorem \ref{thm5.1} follows from the Maslowski and Seidler theory, Proposition 3.1 in \cite{Mas}.
\end{proof}
\begin{remark} We believe that the Theorem \ref{thm5.1} could be extended to the case of equation $\eqref{Equ1.1}_2$ also being forced by a multiplicative noise. For such a system, we given the well-posedness of solution in paper \cite{ZW} with extra boundary condition corresponding to the phenomenon of vesicle membranes. Here, to use the stochastic compactness argument, only minor modification is required, that is, $\phi\in C_t^{\tilde{\alpha}}(
H^{-1})$, $\mathbb{P}$~a.s. for $\tilde{\alpha}$ less than $\frac{1}{2}$ strictly.
\end{remark}

\section{Appendix}

At the beginning, we introduce lemma used frequently for controlling the nonlinear term.
\begin{lemma}\!\!\!\cite[Lemma 3]{simon}\label{lem6.1} Let $r^*=\left\{\begin{array}{ll}
\!\!\frac{dr}{d-r}, {\rm if } ~r<d,\\
\!\!  {\rm any ~finite~nonnegetive~ real~ number,~if}~r=d,\\
\!\! \infty, {\rm if } ~r>d,
\end{array}\right.$ where $d=2,3$ is the dimension.  For $1\leq r\leq \infty, 1\leq s\leq \infty$, if $\frac{1}{r}+\frac{1}{s}\leq 1$, and $\frac{1}{r^*}+\frac{1}{s}=\frac{1}{t}$, $f\in H^{1,r}$ and $g\in H^{-1,s}$, then $fg\in H^{-1,t}$, that is,
$$\|fg\|_{H^{-1,t}}\leq \|f\|_{H^{1,r}}\|g\|_{H^{-1,s}}.$$
\end{lemma}

To establish the compactness property of a family of solutions, the following Aubin-Lions lemma and Skorokhod-Jakubowsk  representation theorem are commonly used. For more background we recommend Corollary 5 of \cite{Simon} and Theorem 2.4 of \cite{Zabczyk}. The Vitali convergence theorem is applied to identify the limit (see for example, Chapter 3 of \cite{Kallenberg}).

\begin{lemma}\!\!\!(The Aubin-Lions Lemma)\label{lem4.5} Suppose that $X_{1}\subset X_{0}\subset X_{2}$ are Banach spaces and $X_{1}$ and $X_{2}$ are reflexive and the embedding of $X_{1}$ into $X_{0}$ is compact.
Then for any $1\leq p<\infty,~ 0<\alpha<1$, the embedding
\begin{equation*}
L^{p}_tX_{1}\cap C^{\alpha}_tX_{2} \hookrightarrow L^{p}_tX_{0},~ L^{\infty}_tX_{1}\cap C^{\alpha}_tX_{2}\hookrightarrow C_t((X_{1})_w)
\end{equation*}
and 
$$L^{p}_tX_{1}\cap W_t^{\alpha, r}X_2\hookrightarrow L^{p}_tX_{0},$$
where $\alpha>0$ if $r\geq p$, is compact, respectively.
\end{lemma}

\begin{theorem}\!\!\!(The Vitali convergence theorem)\label{thm4.1} Let $p\geq 1$, $\{X_n\}_{n\geq 1}\in L^p$ and $X_n\rightarrow X$ in probability. Then, the following are equivalent\\
{\rm (1)}. $X_n\rightarrow X$ in $L^p$;\\
{\rm (2)}. the sequence $|X_n|^p$ is uniformly integrable;\\
{\rm (3)}. $\mathbb{E}|X_n|^p\rightarrow \mathbb{E}|X|^p$.
\end{theorem}

\begin{theorem}\!\!\!(The Skorokhod-Jakubowsk representative theorem)\label{thm4.2} Let $X$ be a topological space. If the set of probability measures $\{\nu_n\}_{n\geq 1}$ on $\mathcal{B}(X)$ is tight, then there exists a probability space $(\Omega, \mathcal{F}, \mathbb{P})$ and a sequence of random variables $u_n, u$ such that theirs laws are $\nu_n$, $\nu$ and $u_n\rightarrow u$, $\mathbb{P}$-a.s. as $n\rightarrow \infty$.
\end{theorem}

To pass the limit in stochastic integral of solution sequence, we take the following result from \cite[Lemma 2.1]{ANR}.
\begin{lemma}\label{lem4.6} Assume that $G_\varepsilon$ is a sequence of $X$-valued $\mathcal{F}_{t}^\varepsilon$ predictable processes such that
$$G_\varepsilon\rightarrow G,~ {\rm in~ probability~ in}~L_t^2L_2(\mathcal{H};X),$$
and the cylindrical Wiener process sequence $\mathcal{W}_\varepsilon$ satisfies
$$\mathcal{W}_\varepsilon\rightarrow \mathcal{W},~ {\rm in ~probability~ in}~C_t\mathcal{H}_0,$$
then,
$$\int_{0}^{t}G_\varepsilon d\mathcal{W}_\varepsilon\rightarrow \int_{0}^{t}G d\mathcal{W}, ~ {\rm in ~probability~ in}~ L_t^2X.$$
\end{lemma}

\begin{lemma}\cite[Lemma 1.1]{gy}(The Gy\"{o}ngy-Krylov Lemma)\label{lem6.4}
Let $X$ be a complete separable metric space and suppose that $\{Y_{n}\}_{n\geq0}$ is a sequence of $X$-valued random variables on a probability space $(\Omega,\mathcal{F},\mathbb{P})$. Let $\{\mathcal{L}_{m,n}\}_{m,n\geq1}$ be the set of joint laws of $\{Y_{n}\}_{n\geq1}$, that is
\begin{equation*}
\mathcal{L}_{m,n}(E):=\mathbb{P}\{(Y_{n},Y_{m})\in E\},~~~E\in\mathcal{B}(X\times X).
\end{equation*}
Then $\{Y_{n}\}_{n\geq1}$ converges in probability if and only if for every subsequence of the joint probability laws $\{\mu_{m_{k},n_{k}}\}_{k\geq1}$, there exists a further subsequence that converges weakly to a probability measure $\mu$ such that
\begin{equation*}
\mathcal{L}\{(Y_1,Y_2)\in X\times X: Y_1=Y_2\}=1.\\
\end{equation*}
\end{lemma}

\begin{proof}[\underline{Proof of Step 3 of Theorem \ref{thm3.1}}]

Invoked by \cite{ANR, gy}, define by $\mathcal{L}_{n,m}$ the joint law
\begin{equation*}
(u_{n}, \phi_{n}, \mu_n;u_{m}, \phi_{m},\mu_m; \mathcal{W})~~~~ {\rm on~ the ~path ~space}~\mathcal{X}=\mathcal{X}_{u}\times \mathcal{X}_{\phi}\times \mathcal{X}_{\mu}\times \mathcal{X}_{u}\times \mathcal{X}_{\phi}\times\mathcal{X}_{\mu} \times \mathcal{X}_{\mathcal{W}},
\end{equation*}
where
$$\mathcal{X}_{u}:=C_t(\mathrm{H}^{-1})\cap L^2_t\mathrm{L}^2_x,~~ \mathcal{X}_{\phi}:=C_t(L_x^2)\cap L^2_tH^1,~ \mathcal{X}_{\mu}:= L^2_tL_x^2,~\mathcal{X}_{\mathcal{W}}:=C_t(\mathcal{H}_0),$$
and $\{u_{n}, \phi_{n};u_{m}, \phi_{m}\}_{n,m\geq 1}$ are two sequences of approximate solutions to system \eqref{E2.1}. Using the same argument as Lemma \ref{lem4.1}, we infer that,
$$ {\rm the~ collection ~of ~joint~ laws}~ \{\mathcal{L}_{m,n}\}_{n,m\geq 1}~{\rm is~ tight~ on} ~\mathcal{X}.$$
For any subsequence $\{\mathcal{L}_{n_{k},m_{k}}\}_{k\geq 1}$, there exists a measure $\mathcal{L}$ such that $\{\mathcal{L}_{n_{k},m_{k}}\}_{k\geq 1}$ converges weakly to $\mathcal{L}$. Applying the Skorokhod-Jakubowsk representation theorem, we have a new probability space $(\widehat{\Omega},\widehat{\mathcal{F}},\widehat{\mathbb{P}})$ and $\mathcal{X}$-valued random variables
\begin{eqnarray*}
(\hat{u}_{n_k}, \hat{\phi}_{n_k}, \hat{\mu}_{n_k};\hat{u}_{m_k}, \hat{\phi}_{m_k}, \hat{\mu}_{m_k}; \widehat{\mathcal{W}}_{k})~ {\rm and}~(\hat{u}_1, \hat{\phi}_{1}, \hat{\mu}_1;\hat{u}_2, \hat{\phi}_{2}, \hat{\mu}_2; \widehat{\mathcal{W}})
\end{eqnarray*}
such that
\begin{eqnarray*}
 &&\widehat{\mathbb{P}}\{(\hat{u}_{n_k}, \hat{\phi}_{n_k}, \hat{\mu}_{n_k};\hat{u}_{m_k}, \hat{\phi}_{m_k}, \hat{\mu}_{m_k}; \widehat{\mathcal{W}}_{k})\in \cdot\}=\mathcal{L}_{n_{k},m_{k}}(\cdot),\\
 &&\widehat{\mathbb{P}}\{(\hat{u}_1, \hat{\phi}_{1}, \hat{\mu}_1;\hat{u}_2, \hat{\phi}_{2}, \hat{\mu}_2; \widehat{\mathcal{W}})\in \cdot\}=\mathcal{L}(\cdot)
\end{eqnarray*}
 and
\begin{eqnarray*}
(\hat{u}_{n_k}, \hat{\phi}_{n_k}, \hat{\mu}_{n_k};\hat{u}_{m_k}, \hat{\phi}_{m_k}, \hat{\mu}_{m_k}; \widehat{\mathcal{W}}_{k})\rightarrow (\hat{u}_1, \hat{\phi}_{1}, \hat{\mu}_1;\hat{u}_2, \hat{\phi}_{2}, \hat{\mu}_2; \widehat{\mathcal{W}}),~~\widehat{\mathbb{P}}\mbox{-a.s.}
\end{eqnarray*}
in the topology of $\mathcal{X}$. Analogously, this argument can be applied to both
\begin{equation*}
(\hat{u}_{n_k}, \hat{\phi}_{n_k},\widehat{\mathcal{W}}_{k}),
~~(\hat{u}_1, \hat{\phi}_{1},\widehat{\mathcal{W}}) \hspace{.3cm} \text{and} \hspace{.3cm}
(\hat{u}_{m_k}, \hat{\phi}_{m_k}, \widehat{\mathcal{W}}_{k}),
~~(\hat{u}_2, \hat{\phi}_{2}, \widehat{\mathcal{W}})
\end{equation*}
to show that $(\hat{u}_1, \hat{\phi}_{1}, \widehat{\mathcal{W}})$ and $(\hat{u}_2, \hat{\phi}_{2}, \widehat{\mathcal{W}})$ are two martingale solutions relative to the same stochastic basis $\widehat{\mathcal{S}}:=(\widehat{\Omega},\widehat{\mathcal{F}},\widehat{\mathbb{P}},\{\widehat{\mathcal{F}}_{t}\}_{t\geq 0},\widehat{\mathcal{W}})$.

In addition, we have $\mathcal{L}_{n,m}\rightharpoonup \mathcal{L}$ where $\mathcal{L}$ is defined by
$$\mathcal{L}(\cdot)=\widehat{\mathbb{P}}\{(\hat{u}_1, \hat{\phi}_{1}, \hat{\mu}_1;\hat{u}_2, \hat{\phi}_{2}, \hat{\mu}_2)\in \cdot\}.$$
The uniqueness implies that $\mathcal{L}\{(u_1, \phi_{1}, \mu_1;u_2, \phi_{2}, \mu_2)\in \mathcal{X}:(u_1, \phi_{1}, \mu_1)=(u_2, \phi_{2}, \mu_2)\}=1$.

The both conditions of Lemma \ref{lem6.4} are verified, we deduce that the sequence $(u_n, \phi_{n}, \mu_n)$ defined on the original probability space $(\Omega,\mathcal{F},\mathbb{P})$ converges $\mathbb{P}$-a.s. in the topology of $\mathcal{X}_{u}\times \mathcal{X}_{\phi}\times\mathcal{X}_{\mu}$ to processes $(u, \phi, \mu)$.
\end{proof}

\smallskip

\section*{Acknowledgments}
 Z. Qiu's research is supported by the CSC under grant No.201806160015.

\end{document}